\documentclass[11pt]{article}
    \usepackage{amsfonts,amsmath,amssymb,graphicx,color}
    \usepackage{amsthm}
    \usepackage{psfrag}
    \usepackage{pdflscape}
    \usepackage{multirow}
    \usepackage{blindtext} 
    \usepackage{enumitem}
    \graphicspath{ {./figures/} }
    \usepackage{url}
    \usepackage{soul}

    \newcommand{\tr}{\mbox{tr}}

    \numberwithin{figure}{section}
    \numberwithin{table}{section}

    \def\E{\mathbb{E}}
    
    \newcommand{\bequ}{\begin{equation}}     \newcommand{\eequ}{\end{equation}}
    \newcommand{\benn}{\begin{equation*}}    \newcommand{\eenn}{\end{equation*}}
    \newcommand{\bbma}{\begin{bmatrix}}      \newcommand{\ebma}{\end{bmatrix}}

    \newcommand{\covar}{\mbox{Cov}}
    \newcommand{\trace}{\mbox{tr}}

    \newcommand{\R}{\mathbb{R}}

    \newcommand{\bsub}{\begin{subequations}}
    	\newcommand{\esub}{\end{subequations}}

    \newtheorem{thm}{Theorem}[section]

    \newtheorem{lem}[thm]{Lemma}

    \numberwithin{equation}{section}

    \newcommand{\comment}[1]{}
    
    \newcommand{\be}{\begin{equation}}
    \newcommand{\ee}{\end{equation}}
    
    \newcommand{\bea}{\begin{eqnarray}}
    \newcommand{\eea}{\end{eqnarray}}
    \newcommand{\beqa}{\begin{eqnarray}}
    \newcommand{\eeqa}{\end{eqnarray}}
    \newcommand{\beann}{\begin{eqnarray*}}
    	\newcommand{\eeann}{\end{eqnarray*}}
    
    \newcommand{\bmat}{\left[ \begin{array}}
    	\newcommand{\emat}{\end{array} \right]}
    
    \newcommand{\beq}{\begin{equation}}
    \newcommand{\eeq}{\end{equation}}

    \newcommand{\bproof}{\begin{description} \item[{\it Proof}.] ~ }
    	\newcommand{\eproof}{\hspace*{\fill}$\Box$\medskip \end{description}}

    \newcommand{\defeq}{\stackrel{\rm def}{=}}

    \newcounter{algo}[section]

    		\newcounter{prog}[section]

   				\title{Exact and Inexact Subsampled Newton Methods for Optimization}
   				\author{       
   					Raghu Bollapragada\thanks{Department of Industrial Engineering and Management Sciences, Northwestern University, 
   						Evanston, IL, USA. This author was supported by the Office of Naval Research award N000141410313. }
   					\and     
   					Richard Byrd \thanks{Department of Computer Science, University of Colorado,
   						Boulder, CO, USA.  This author was supported by National Science Foundation
   						grant DMS-1620070.}
   					\and
   					Jorge Nocedal \thanks{Department of Industrial Engineering and Management Sciences, Northwestern University, 
   						Evanston, IL, USA.  This author was supported by Department 
   						of Energy grant DE-FG02-87ER25047 and National Science Foundation grant DMS-1620022.} 
   				}

   				\date{\today}
   				
\begin{document}

   					\maketitle
   					
   					\begin{abstract}
   						The paper studies the solution of stochastic optimization problems in which approximations to the gradient and Hessian are  obtained through subsampling. We first consider Newton-like methods that employ these approximations and discuss how to coordinate the accuracy in the gradient and Hessian to yield a superlinear rate of convergence in expectation. The second part of the paper analyzes an inexact Newton method that solves linear systems approximately using the conjugate gradient (CG) method, and that samples the Hessian and not the gradient (the gradient is assumed to be exact).  We provide a complexity analysis for this method based on the properties of the CG iteration and the quality of the Hessian approximation, and compare it with a method that employs a stochastic gradient iteration instead of the CG method.  We report preliminary numerical results that illustrate the performance of inexact subsampled Newton methods on machine learning applications based on logistic regression.
   					\end{abstract}

   					\bigskip\bigskip
   					\section{Introduction}
   					\setcounter{equation}{0}
   					In this paper we study subsampled Newton methods for stochastic optimization. These methods employ approximate  gradients and  Hessians, through sampling, in order to achieve efficiency and scalability. Additional economy of computation is obtained by solving  linear systems inexactly at every iteration, i.e., by implementing  inexact Newton methods. We study the convergence properties of (exact) Newton methods that approximate both the Hessian and gradient, as well as the complexity of inexact Newton methods that subsample only the Hessian and use the conjugate gradient method to solve linear systems.
   					
   					The optimization problem of interest arises in machine learning applications but with appropriate modifications is relevant to  other stochastic optimization settings including simulation optimization \cite{amaran2014simulation, fu2015handbook}. We state the  problem as
   					\begin{equation}\label{prob}
   					\min_{w \in \R^{d}}   F(w) = \int f(w; x,y) dP( x,y),
   					\end{equation}
   					where $f$ is the composition of a prediction function (parametrized by a vector $w \in \R^d$) and a smooth loss function, and $(x,y)$ are the input-output pairs with joint probability distribution $P(x,y)$.	We call $F$  the \emph{expected risk}.
   					
   					In practice, one does not have complete information about $P(x, y)$, and therefore one works with data $\{(x^i,y^i)\}$ drawn  from the distribution $P$. One may view an optimization algorithm as being applied directly to the expected risk \eqref{prob}, or given $N$ data
   					points,  as being applied to the \emph{empirical risk}  
   					\[
   					R(w) = \frac{1}{N} \sum_{i =1}^N f(w;x^i,y^i) .
   					\]
   					To simplify the notation, we define 
   					$
   					F_i(w) = f(w;x^i,y^i),
   					$
   					 and thus we write
   					\begin{equation}\label{probrr}
   					R(w) = \frac{1}{N} \sum_{i =1}^N F_i(w)
   					\end{equation}     
   					in the familiar finite-sum form arising in many applications beyond machine learning \cite{Bert95}. In this paper, we consider both objective functions, $F$ and $R$.
   					
   					The iteration of a subsampled Newton method for minimizing $F$ is given by
   					\begin{equation} \label{Newton-CG-intr}
   					w_{k+1} = w_k + \alpha_k p_k ,
   					\end{equation}
   					where $p_k$ is the solution of the \emph{Newton equations}
   					\begin{equation} \label{lins}
   					\nabla^2F_{S_k}(w_k)p_k = -\nabla F_{ X_k}(w_k) .
   					\end{equation}
   					Here, the subsampled gradient and Hessian are defined as
   					\begin{equation}   \label{subs}
   					\nabla F_{X_k}(w_k)=\frac{1}{|X_k|}\sum_{i \in X_k} \nabla F_i(w_k),  \qquad \nabla^2F_{S_k}(w_k)= \frac{1}{|S_k|}\sum_{i \in S_k} \nabla^2 F_i(w_k),
   					\end{equation}
   					where the sets $ X_k, S_k \subset\{1, 2, \ldots \}$ index sample points $ (x^i, y^i)$ drawn at random from the distribution $P$. We refer to $X_k$ and $S_k$ as the gradient and Hessian samples --- even though they only refer to indices of the samples. The choice of the sequences $\{X_k\}$ and $\{S_k\}$  gives rise to  distinct algorithms, and our  goal is to identify the most promising instances, in theory and in practice.
   					
   					In the first part of the paper, we consider Newton methods  in which the linear system \eqref{lins} is solved \emph{exactly}, and  we identify conditions on $\{S_{k}\}$ and $\{X_{k}\}$ under which linear or superlinear rates of convergence are achieved. Exact Newton methods are practical when the number of variables $d$ is not too large, or when the structure of the problem allows a direct factorization of the Hessian $ \nabla^2F_{S_k}$  in time linear in the number of variables $d$. 
   					
   					For most large-scale applications, however, forming the Hessian $\nabla^2F_{S_k}(w_k)$ and solving the linear system \eqref{lins} accurately is prohibitively expensive, and one has to compute an approximate solution in $O(d)$ time using an iterative linear solver that only requires Hessian-vector products (and not the Hessian itself).  Methods based on this strategy are sometimes called \emph{inexact Hessian-free Newton methods}. In the second part of the paper, we study how to balance the accuracy of this linear solver with the sampling technique used for the Hessian so as to obtain an efficient algorithm.  In doing so, we pay particular attention to the properties of two iterative linear solvers for \eqref{lins}, namely the conjugate gradient method (CG) and a stochastic gradient iteration (SGI). 
   					
   					It is generally accepted that in the context of deterministic optimization and when the matrix in \eqref{lins} is the exact Hessian, the CG method is the iterative solver of choice due to its notable convergence properties.  However, subsampled Newton methods provide a different setting where other iterative methods could be more effective.  For example, solving \eqref{lins} using a stochastic gradient iteration (SGI) has the potential advantage that the sample $S_k$ in \eqref{lins} can be changed at every iteration, in contrast with the Newton-CG method where it is essential to fix the sample $S_k$ throughout the CG iteration.
   					
   					It is then natural to ask: which of the two methods, Newton-CG or Newton-SGI, is more efficient, and how should this be measured? Following \cite{agarwal2016second}, we phrase this question by asking how much computational effort does each method require in order to yield a given local rate of convergence --- specifically, a linear rate with convergence constant of 1/2.

   					\medskip
   					\emph{N.B.} While this paper was being written for publication,  Xu et al.  \cite{xu2016sub} posted a paper that gives  complexity results that, at first glance, appear to be very similar to the ones presented in sections~\ref{inexactn} and \ref{complex} of this paper. They also compare the complexity of Newton-CG and Newton-SGI  by estimating the amount of work required to achieve linear convergence rate with constant 1/2. However, their analysis establishes convergence rates in probability, \emph{for one iteration}, whereas in this paper we prove convergence in expectation for the sequence of iterates. The two papers  differ in other respects. We give (in section~\ref{subnewton}) conditions on $S_k$ and $X_k$ that guarantee superlinear convergence, whereas  \cite{xu2016sub}  studies the effect of non-uniform sampling.

   					\subsection{Related Work}
   					Subsampled gradient  and Newton methods have recently received much attention. Friedlander and Schmidt \cite{friedlander2012hybrid} and Byrd et al. \cite{byrd2012sample} analyze the rate at which $X_k$ should increase so that the subsampled gradient  method (with fixed steplength) converges linearly to the solution of strongly convex problems.  Byrd et al. \cite{byrd2012sample} also provide work-complexity bounds for their method and report  results of experiments with  a subsampled Newton-CG method, whereas Friedlander and Schmidt \cite{friedlander2012hybrid}  study the numerical performance of L-BFGS using  gradient sampling techniques.  Martens \cite{Martens10} proposes  a subsampled Gauss-Newton method for training deep neural networks, and focuses on the choice of the regularization parameter. None of these  papers provide a convergence analysis for subsampled Newton methods.
   					
   					Pasupathy et al. \cite{2014pasglyetal} consider sampling rates in a more general setting. Given a deterministic algorithm --- that could have linear, superlinear or quadratic convergence --- they analyze the stochastic analogue that subsamples the gradient, and  identify optimal sampling rates for several families of algorithms.
   					Erdogdu and Montanari \cite{erdogdu2015convergence} study a Newton-like method where the Hessian approximation is obtained by first subsampling the true Hessian and then computing a truncated eigenvalue decomposition. Their method employs a full gradient and is designed for problems where  $d$ is not so large that the cost of iteration, namely  $O(Nd + |S| d^2)$, is affordable.  
   					
   					Roosta-Khorasani and Mahoney  \cite{roosta2016sub1,roosta2016sub}  derive global and local convergence rates for subsampled Newton methods with various sampling rates used for gradient and Hessian approximations. Our convergence results are similar to theirs, except that they employ matrix concentration inequalities \cite{tropp_computational_2010} and state progress at each iteration in probability --- whereas we do so in expectation.  \cite{roosta2016sub} goes beyond other studies in the literature in that they assume the objective function is strongly convex but the individual component functions are only weakly convex, and show how to ensure that the subsampled matrices are invertible (in probability).
   					
   					Pilanci and Wainwright \cite{pilanci2015newton}  propose a Newton sketch method that approximates the Hessian via random projection matrices while employing the full gradient of the objective. The best complexity results are obtained when the projections are performed using the randomized Hadamard transform.  This method requires access to the square root of the true Hessian, which, for generalized linear models entails access to the entire dataset at each iteration.  
   					
   					Agarwal et al. \cite{agarwal2016second} study a Newton method that aims to compute unbiased estimators of the inverse Hessian and that achieves a linear time complexity in the number of variables. Although they present the method as one that computes a power expansion of the Hessian inverse, they show that it can also be interpreted as a subsampled Newton method where the step computation is performed inexactly using a stochastic gradient iteration. As mentioned above, Xu et al.  \cite{xu2016sub} study the convergence properties of inexact subsampled Newton methods that employ non-uniform sampling.


   					\bigskip\noindent  
   					\subsection{Notation} We denote the variables of the optimization problem by $w \in \R^d$, and a minimizer of the objective $F$ as $w^*$. Throughout the paper, we use  $\| \cdot \|$ to represent the $\ell_2$ vector norm or its induced matrix norm. The notation $A \preceq B$ means that $B-A$ is a symmetric and positive semi-definite matrix.   
    		
\section{Subsampled Newton Methods}    \label{subnewton}
The problem under consideration in this section is the minimization of expected risk \eqref{prob}.
The general form of an (exact) subsampled Newton method for this problem is given by
\begin{equation} \label{iterate}
w_{k+1} = w_k - \alpha_{k}\nabla^2F_{S_k}^{-1}(w_k)\nabla F_{X_k}(w_k) ,
\end{equation}
where $\nabla^2F_{S_k}(w_k)$ and $\nabla F_{X_k}(w_k)$ are defined in \eqref{subs}.  We assume that the sets $\{X_k\}, \{S_k\} \subset \{1, 2, \ldots \}$ are chosen independently  (with or without replacement). 

The steplength $\alpha_{k}$ in \eqref{iterate} is chosen to be a constant or  computed by a backtracking line search; we do not consider the case of diminishing steplengths, $\alpha_{k} \rightarrow 0$, which leads to a sublinear rate of convergence.
It is therefore clear that, in order to obtain convergence to the solution, the sample size $X_k$ must increase, and in order to obtain a rate of convergence that is faster than linear, the Hessian sample $S_k$ must also increase. We now investigate the rules for controlling those samples.
We assume in our analysis that the objective function $F$ is strongly convex. For the sake of brevity, we make the stronger assumption that each of the component functions $F_i$ is strongly convex, which can be ensured through regularization. (This stipulation on the $F_{i}$ is relaxed e.g. in \cite{roosta2016sub}).
The main set of assumptions made in this paper is as follows. \\

\noindent   {\bf Assumptions I}
\begin{enumerate}
	\item [A1] {\bf(Bounded Eigenvalues of Hessians)} The function $F$ is twice continuously differentiable and any subsampled Hessian is positive definite with eigenvalues lying in a positive interval (that depends on the sample size). That is, for any integer $\beta$ and any set $S \subset \{1,2,\ldots \}$ with $|S|=\beta$, there exist positive constants $ \mu_{\beta}, L_{\beta}$ such that
	\begin{equation} \label{eigen}
	\mu_{\beta} I   \preceq \nabla^2 F_{S}(w) \preceq L_{\beta} I, \quad \forall w \in \R^d.
	\end{equation}
	Moreover, there exist  constants $\bar \mu, \bar L$ such that 
	\begin{equation}  \label{bars}
	0 < \bar \mu \leq  \mu_\beta \quad\mbox{and} \quad L_\beta \leq \bar{L} < \infty, \quad\mbox{for all  } \beta \in \mathbb{N}.
	\end{equation}
	The smallest and largest eigenvalues corresponding to the  objective $F$ are denoted by $\mu, L$  (with $0< \mu, L < \infty)$,
	 i.e.,
	\begin{equation}  \label{truel}
	\mu I \preceq \nabla^2 F(w) \preceq L I , \quad \forall w \in \R^d.
	\end{equation}
	
	\item [A2] {\bf(Bounded variance of sample gradients)} The trace of the covariance matrix of the individual sample gradients is uniformly bounded, i.e., there exists a constant $v$ such that 
	\begin{equation} \label{Var}
	\tr \left(\covar \left(\nabla F_i(w) \right)\right) \leq v^2, \qquad \forall w \in \R^d. 
	\end{equation}
	\item [A3] {\bf(Lipschitz Continuity of Hessian)} The Hessian of the objective function $F$ is Lipschitz continuous, 
	i.e., there is a constant $M >0$ such that 
	\begin{equation} \label{liphes}
	\|\nabla^2 F(w) - \nabla^2 F(z)\| \leq M\|w - z\|, \qquad \forall w,z  \in \R^d .
	\end{equation}
	\item [A4] {\bf(Bounded Variance of Hessian components)} 
	There is a constant $\sigma$ such that, for all component Hessians, we have 
	\begin{equation}   \label{bvhes}
	\|\E[(\nabla^2 F_i(w) - \nabla^2 F(w))^2]\| \leq \sigma^2, \qquad \forall w \in \R^d.
	\end{equation}    
\end{enumerate}

\medskip\noindent
We let $w^*$ denote the unique minimizer of $F$. 

\subsection{Global Linear Convergence}
We now show that for the Newton method \eqref{iterate} to enjoy an R-linear rate of convergence  the gradient sample size must be increased at a geometric rate, i.e., $|X_{k}|=\eta^k$  for some $\eta >1$. On the other hand, the subsampled Hessian need not be accurate, and thus  it suffices to keep samples $S_{k}$ of constant size, $|S_k|=\beta \geq 1$. The following result, in which the steplength $\alpha_k$ in \eqref{iterate} is constant, is a simple extension of well known results  (see, e.g. \cite{byrd2012sample, friedlander2012hybrid, 2014pasglyetal}), but we include the proof for the sake of completeness. We assume that the set $X_k$ is drawn uniformly at random so that at every iteration $ \E[ \nabla F_{X_k}(w_k)] = \nabla F(w_k)$. We also assume that the sets $X_k$ and $S_k$ are chosen independently of each other.

\begin{thm} \label{thmlin}
	Suppose that Assumptions A1-A2 hold. Let $\{w_k\}$ be the iterates generated by  iteration \eqref{iterate} with any $w_0$, where $|X_{k}|=\eta^k$ for some $\eta >1$, and $|S_k|=\beta \geq 1$ is constant. Then, if the steplength satisfies $\alpha_k = \alpha = \frac{\mu_{\beta}}{L} $, we have that   
	\begin{equation} \label{linear}
	\E[F(w_k) - F(w^*)] \leq C\hat{\rho}^k ,
	\end{equation}
	where 
	\begin{equation} \label{parsl}
	C= \max\left\{F(w_0) - F(w^*), \frac{v^2L_{\beta}}{\mu\mu_{\beta}}\right\}\quad \mbox{ and } \quad \hat{\rho} = \max\left\{1 - \frac{\mu\mu_{\beta}}{2LL_{\beta}},\frac{1}{\eta} \right\}.
	\end{equation}
\end{thm}
\begin{proof}
	Let $\E_k$ denote the conditional expectation at iteration $k$ for all possible sets  $X_k$. Then for any given $S_k$,
	\begin{align*}
	\E_k\left[F(w_{k+1})\right] &\leq   F(w_k)  - \alpha \nabla F(w_k)^T\nabla^2 F_{S_k}^{-1} (w_k)\E_k[\nabla F_{X_k}(w_k)] + \frac{L\alpha^2}{2}\E_k[\|\nabla^2 F_{S_k}^{-1} (w_k)\nabla F_{X_k}(w_k)\|^2]  \nonumber \\
	&=   F(w_k)  - \alpha \nabla F(w_k)^T\nabla^2 F_{S_k}^{-1} (w_k)\nabla F(w_k) + \frac{L\alpha^2}{2} \| \E_k[ \nabla^2 F_{S_k}^{-1} (w_k)\nabla F_{X_k}(w_k)] \|^2   \\  
	& \quad + \frac{L\alpha^2}{2}  \E_k[ \|\nabla^2 F_{S_k}^{-1} (w_k)\nabla F_{X_k}(w_k) - \E_k[ \nabla^2 F_{S_k}^{-1} (w_k)\nabla F_{X_k}(w_k)] \|^2 ]  \nonumber \\
	&=   F(w_k)  - \alpha \nabla F(w_k)^T \left(\nabla^2 F_{S_k}^{-1} (w_k) -\frac{L\alpha}{2}  \nabla^2 F_{S_k}^{-2} (w_k)\right)\nabla F(w_k)   \\  
	& \quad + \frac{L\alpha^2}{2}  \E_k[ \| \nabla^2 F_{S_k}^{-1} (w_k)\nabla F_{X_k}(w_k) -  \nabla^2 F_{S_k}^{-1} (w_k)\nabla F(w_k)  \|^2  ] \nonumber \\
	&\leq   F(w_k)  - \alpha \nabla F(w_k)^T \nabla^2 F_{S_k}^{-1/2} (w_k) \left(I -\frac{L\alpha}{2}  \nabla^2 F_{S_k}^{-1} (w_k)\right) \nabla^2 F_{S_k}^{-1/2}(w_k)\nabla F(w_k)   \\  
	& \quad + \frac{L\alpha^2}{2\mu_{\beta}^2}  \E_k[ \| \nabla F_{X_k}(w_k) -  \nabla F(w_k)  \|^2  ] .\nonumber \\
	\end{align*}
	Now,  $\{ \nabla^2 F_{S_k}^{-1}\}$ is a sequence of random variables but we can bound its eigenvalues from above and below. Therefore, we can use these eigenvalue bounds as follows:
	\begin{align*}
	\E_k\left[F(w_{k+1})\right]
	&\leq   F(w_k)  - \alpha  (1 -\frac{L\alpha}{2\mu_{\beta}} ) (1/L_{\beta}) \|\nabla F(w_k)  \|^2 + \frac{L\alpha^2}{2\mu_{\beta}^2}  \E_k[ \| \nabla F_{X_k}(w_k) -  \nabla F(w_k)  \|^2  ] \nonumber \\
	&\leq   F(w_k)  - \frac{\mu_{\beta}}{2LL_{\beta} } \|\nabla F(w_k)  \|^2 + \frac{1}{2L}  \E_k[ \| \nabla F_{X_k}(w_k) -  \nabla F(w_k)  \|^2  ] \nonumber \\
	& \leq F(w_k) - \frac{\mu\mu_{\beta}}{LL_{\beta}}(F(w_k) - F(w^*)) + \frac{1}{2L} \E_k[\|\nabla F(w_k) - \nabla F_{X_k}(w_k)\|^2], \nonumber  
	\end{align*}
	where the last inequality follows from the fact that, for any $\mu-$strongly convex function, $\| \nabla F(w_k) \|^2 \geq 2\mu (F(w_k) - F(w^*))$.
	Therefore, we get
	\begin{align}
	\E_k\left[F(w_{k+1}) - F(w^*)\right] &\leq \left(1 - \frac{\mu\mu_{\beta}}{LL_{\beta}}\right)(F(w_k) - F(w^*)) + \frac{1}{2L} \E_k[\|\nabla F(w_k) - \nabla F_{X_k}(w_k)\|^2 ] . \label{garden}
	\end{align}
	%
	Now, by Assumption A2 we have that
	\begin{align}
	\E_k[\|\nabla F(w_k) - \nabla F_{X_k}(w_k)\|^2] &= \E_k\left[ \tr\left(\left(\nabla F(w_k) - \nabla F_{X_k}(w_k)\right)\left(\nabla F(w_k) - \nabla F_{X_k}(w_k)\right)^T\right)\right] \nonumber \\
	&= \tr\left(\covar\left(\nabla F_{X_k}(w_k)\right)\right) \nonumber \\
	&= \tr\left(\covar\left(\frac{1}{|X_k|}\sum_{i\in X_k} \nabla F_i(w_k)\right)\right) \nonumber \\
	& \leq \frac{1}{|X_k|} \tr(\covar(\nabla F_i(w_k))) \nonumber \\
	& \leq \frac{v^2}{|X_k|} .\label{abound}
	\end{align}
	Substituting this inequality in \eqref{garden}, we obtain
	\begin{equation} \label{linear-lemma}
	\E_k\left[F(w_{k+1}) - F(w^*)\right] \leq \left(1 - \frac{\mu\mu_{\beta}}{LL_{\beta}}\right)(F(w_k) - F(w^*)) + \frac{v^2}{2L|X_k|} .
	\end{equation}
	We use induction for the rest of the proof, and to this end we recall the definitions of $C$ and $\hat{\rho}$.
	Since $\E[F(w_0) - F(w^*)] \leq C$, inequality \eqref{linear} holds for $k$=0. Now, suppose that \eqref{linear} holds for some $k$. Then by \eqref{linear-lemma}, the condition $|X_k| = \eta^k$, the definition of $\hat{\rho}$, and taking expectations, we have
	\begin{align*}
	\E[F(w_{k+1}) - F(w^*)] &\leq \left(1 - \frac{\mu\mu_{\beta}}{LL_{\beta}}\right) C\hat{\rho}^k + \frac{v^2}{2L|X_k|} \\
	&= C\hat{\rho}^k\left(1 - \frac{\mu\mu_{\beta}}{LL_{\beta}}+ \frac{v^2}{2LC(\hat{\rho}\eta)^k}\right) \\
	&\leq C\hat{\rho}^k\left(1 -  \frac{\mu\mu_{\beta}}{LL_{\beta}} + \frac{v^2}{2LC}\right) \\
	&\leq C\hat{\rho}^k\left(1 -  \frac{\mu\mu_{\beta}}{LL_{\beta}} +  \frac{\mu\mu_{\beta}}{2LL_{\beta}}\right) \\
	&\leq C\hat{\rho}^k\left(1 -   \frac{\mu\mu_{\beta}}{2LL_{\beta}}\right) \\
	&\leq C\hat{\rho}^{k+1} .
	\end{align*}       
\end{proof}  

If one is willing to increase the Hessian sample size as the iterations progress, then one can achieve a faster rate of convergence, as discussed next.

\subsection{Local Superlinear Convergence.}
We now discuss how to design the subsampled Newton method (using a unit stepsize) so as to obtain superlinear convergence in a neighborhood of the solution $w^{*}$. This question is most interesting when the objective function is given by the expectation \eqref{prob} and the indices $i$ are chosen from an  infinite set according to a probability distribution $P$. We will show that the sample size used for gradient estimation should increase at a rate that is \emph{faster than geometric},  i.e. $|X_k| \geq \eta_k^k$, where $\{\eta_k\} >1$ is an increasing sequence,  whereas sample size used for Hessian estimation should increase at any rate such that  $|S_k| \geq |S_{k-1}|$ and $\lim\limits_{k \rightarrow \infty}|S_k|=\infty$. 

We begin with the following result that identifies  three quantities that drive the iteration. Here  $\E_k$ denotes the conditional expectation at iteration $k$ for all possible sets  $X_k$ and $S_k$.
\begin{lem} \label{Terms} Let $\{w_k\}$ be the iterates generated by algorithm \eqref{iterate} with $\alpha_k=1$, and suppose that assumptions A1-A3 hold. Then for each $k$, 
	\begin{equation}   \label{linquad}
	\E_k[\|w_{k+1} - w^*\|] \leq \frac{1}{\mu_{|S_k|}}\left[\frac{M}{2}\|w_k - w^*\|^2 + \E_k\left\|\left(\nabla^2 F_{S_k}(w_k) - \nabla^2 F(w_k)\right)(w_k - w^*)\right\| + \frac{v}{ \sqrt{|X_k|}}\right] .
	\end{equation}
\end{lem}
\begin{proof}
	We have that the expected distance to the solution after the $k$-th step is given by 
	\begin{align}
	\E_k[\|w_{k+1} - w^*\|] &= \E_k [\|w_k - w^* - \nabla^2 F^{-1}_{S_k}(w_k)\nabla F_{X_k}(w_k)\|]  \label{bterm} \\
	&= \E_k \left[\left\|  \nabla^2 F_{S_k}^{-1}(w_k) \left( \nabla^2 F_{S_k}(w_k)(w_k - w^*) - \nabla F(w_k)- 
	\nabla F_{X_k}  (w_k) + \nabla F(w_k) \right) \right\|\right] \nonumber \\
	&\leq \frac{1}{\mu_{|S_k|}}  \E_k\left[ \left\| \left( \nabla^2 F_{S_k}(w_k)- \nabla^2 F(w_k) \right) (w_k - w^*)  
	+ \nabla^2 F(w_k)(w_k - w^*) - \nabla F(w_k) \right\|\right] \nonumber \\
	&  \quad + \frac{1}{\mu_{|S_k|}} \E_k \left[\left\|    \nabla F_{X_k}(w_k) -\nabla F(w_k) \right\|\right]  . \nonumber 
	\end{align}
	Therefore,
	\begin{align}
	\E_k[\|w_{k+1} - w^*\|] &\leq \underbrace{\frac{1}{\mu_{|S_k|}}\|\nabla^2 F(w_k) (w_k -w^*)- \nabla F(w_k)\|}_\text{Term 1} \nonumber \\ 
	 &+ \underbrace{\frac{1}{\mu_{|S_k|}}\E_k\left[\left\|\left(\nabla^2 F_{S_k}(w_k) - \nabla^2 F(w_k)\right)(w_k - w^*)\right\|\right]}_\text{Term 2  \nonumber }\\
	&+ \underbrace{\frac{1}{\mu_{|S_k|}}\E_k[\|\nabla F_{X_k}(w_k) - \nabla F(w_k)\|]}_\text{Term 3}.  \label{terms}
	\end{align}
	For Term~1, we have by Lipschitz continuity of the Hessian \eqref{liphes},
	
	\begin{align*}
	\frac{1}{\mu_{|S_k|}}\|\nabla^2 F(w_k)(w_k - w^*) - \nabla F(w_k)\| 
	& \leq \frac{1}{\mu_{|S_k|}}\|w_k - w^*\|\left\|\int_{t=0}^{1}[\nabla^2 F(w_k) - \nabla^2 F(w_k + t(w^* - w_k))]dt\right\|\\
	& \leq \frac{1}{\mu_{|S_k|}}\|w_k - w^*\|\int_{t=0}^{1}\|\nabla^2 F(w_k) - \nabla^2 F(w_k + t(w^* - w_k))dt\|\\
	&\leq \frac{1}{\mu_{|S_k|}}\|w_k - w^*\|^2\int_{t=0}^{1}Mtdt \\
	& = \frac{M}{2\mu_{|S_k|}}\|w_k - w^*\|^2 .
	\end{align*}
	Term~3 represents the error in the gradient approximation. By Jensen's inequality we have that,
	\begin{align*}
	\left(\E_k[\|\nabla F(w_k) - \nabla F_{X_k}(w_k)\|]\right)^2 &\leq \E_k[\|\nabla F(w_k) - \nabla F_{X_k}(w_k)\|^2].  \nonumber \\
	\end{align*}
	We have shown in \eqref{abound} that $\E_k[\|\nabla F_{X_k}(w_k) - \nabla F(w_k)\|^2] \leq v^2 / |X_k|$, which concludes the proof.
\end{proof}

Let us now consider Term~2 in \eqref{linquad}, which represents the error due to Hessian subsampling. In order to prove convergence, we need to bound this term as a function of the Hessian sample size $|S_k|$. The following lemma shows that this error is inversely related to the square root of the sample size. 
\begin{lem} \label{accuh}
	Suppose that  Assumptions A1 and A4 hold. Then 
	\begin{equation} 
	\E_k[\|\left(\nabla^2 F_{S_k}(w_k) - \nabla^2 F(w_k) \right)(w_k - w^*)\|] \leq \frac{\sigma}{\sqrt{|S_k|}}\|w_k - w^*\| ,
	\end{equation}
	where $\sigma$ is defined in Assumptions A4.
\end{lem} 
\begin{proof}
	Let us define $Z_S = \nabla^2 F_S(w)(w - w^*)$ and $Z=\nabla^2 F(w)(w - w^*)$, so that
	\begin{align*}
	\|\left(\nabla^2 F_{S}(w) - \nabla^2 F(w) \right)(w - w^*)\| = \|Z_S - Z\|.
	\end{align*}
	(For convenience we  drop the iteration index $k$ in this proof.)
	We also write $Z_S = \frac{1}{|S|}\sum Z_i$, where $Z_i = \nabla^2 F_i(w)(w - w^*)$, and note that each $Z_i$ is independent. By Jensen's inequality we have,
	\begin{align*}
	(\E[\|Z_S - Z\|])^2 &\leq \E[\|Z_S - Z\|^2] \\
	&= \E\left[\tr\left((Z_S - Z)(Z_S - Z)^T\right)\right] \\
	&=\trace(\covar(Z_S))\\
	&=\trace\left(\covar\left(\frac{1}{|S|}\sum_{i\in S} Z_i\right)\right)\\
	& \leq \frac{1}{|S|}\trace(\covar(Z_i))  \\
	& = \frac{1}{|S|}\trace\left(\covar(\nabla^2 F_i(w)(w - w^*))\right) .
	\end{align*}
	Now, we have by \eqref{bvhes} and \eqref{truel}
	\begin{align*}
	\trace \left(\covar(\nabla^2F_i(w)(w - w^*))\right) &= (w - w^*)^T\E[(\nabla^2 F_i(w) - \nabla^2 F(w))^2](w - w^*) \\
	 & \leq \sigma^2\|w - w^*\|^2 .
	\end{align*}
\end{proof} 
We note that in the literature on subsampled Newton methods \cite{erdogdu2015convergence, roosta2016sub} it is common to use matrix concentration inequalities to measure the accuracy of Hessian approximations. In Lemma \ref{accuh}, we measure instead the error along the vector $w_k-w^*$, which gives a more precise estimate.

Combining Lemma~\ref{Terms} and Lemma \ref{accuh}, and recalling \eqref{bars}, we obtain the following linear-quadratic bound
\begin{equation} \label{threet}
\E_k[\|w_{k+1} - w^*\|] \leq \frac{M}{2\bar{\mu}}\|w_k - w^*\|^2 + \frac{\sigma\|w_k - w^*\|}{\bar{\mu}\sqrt{|S_k|}} + 
\frac{v}{\bar{\mu} \sqrt{|X_k|}}.
\end{equation}
It is clear that in order to achieve a superlinear convergence rate, it is not sufficient to increase the sample size $|X_k|$  at a geometric rate, because that would decrease the last term in \eqref{threet} at a linear rate; thus $|X_k|$  must be increased at a rate that is faster than geometric.  From the middle term we see that sample size $|S_k|$  should also increase,  and can do so at any rate, provided $|S_k| \rightarrow \infty$. To bound the first term, we introduce the following assumption on the second moments of the distance  of iterates to the optimum. \\

\noindent   {\bf Assumption II} 
\begin{enumerate}
	\item [B1] {\bf(Bounded Moments of Iterates)} There is a constant $\gamma >0 $ such that for any iterate $w_k$ generated by algorithm \eqref{iterate} we have
	\begin{equation} \label{b1}
	\E[||w_k-w^*||^2] \leq \gamma (\E[||w_k - w^*||])^2.
	\end{equation}	   
\end{enumerate}

This is arguably a restrictive assumption on the variance of the error norms, but
we note that it is imposed on non-negative numbers.	

\begin{thm}[{\bf Superlinear convergence}] \label{superlinear}
	Let $\{w_k\}$ be the iterates generated by algorithm \ref{iterate} with stepsize $\alpha_k=\alpha=1$. Suppose that Assumptions A1-A4 and B1 hold and that for all~$k$:
	\begin{enumerate}
		\item[(i)] $|X_k| \geq |X_0|\,\eta^k_k ,\ $ with $\ |X_0| \geq\left(\frac{6v\gamma M}{\bar{\mu}^2}\right)^2$, 
		$\eta_k > \eta_{k-1}$, \, $\eta_k \rightarrow \infty$, and \, $ \eta_1 > 1.$ 
		\item[(ii)]  $|S_k| > |S_{k-1}|, \ $  with $\lim\limits_{k\rightarrow \infty} |S_k| = \infty$, and \, $|S_0| \geq \left(\frac{4\sigma}{\bar \mu}  \right)^2$. 
	\end{enumerate}
	Then, if the starting point satisfies
	\begin{equation*}
	\|w_0 - w^*\| \leq \frac{\bar{\mu}}{3\gamma M},
	\end{equation*}
	we have that $\E[\|w_k  - w^*\|] \rightarrow 0$ at an R-superlinear rate, i.e., there is a positive sequence $\{\tau_k\}$ such that
	\[
	\E[ \| w_k -w^* \|] \leq \tau_k \quad \mbox{and} \quad \tau_{k+1}/\tau_k \rightarrow 0.
	\]
\end{thm}
\begin{proof}
	We establish the result by showing that, for all $k$, 
	\begin{equation}  \label{rlinear}
	\E[\|w_k -w^*\|] \leq \frac{\bar{\mu}}{3\gamma M}\tau_k,
	\end{equation}
	where
	\begin{equation}   \label{lots}
	\tau_{k+1} = \max\left\{{\tau_k\rho_k, \eta_{k+1}^{-(k+1)/4}}\right\} , \quad \tau_0 = 1, \quad   \rho_k =  \frac{\tau_k}{6} + \frac{1}{4}\sqrt{\frac{|S_0|}{|S_k|}} +
	\frac{1}{2\eta^{k/4}_k} .
	\end{equation}
	We use induction to show \eqref{rlinear}.  Note that the base case, $k=0$, is trivially satisfied. Let us assume that the result is true for iteration $k$, so that
	\[
	\frac{3\gamma M}{\bar \mu} \E [\| w_k - w_* \|] \leq \tau_k.
	\]
	Let us now consider iteration $k+1$. Using  \eqref{threet}, the bounds for the sample sizes given in conditions (i)--(ii), \eqref{lots} and \eqref{b1}, we get 
	\begin{align*}
	\E[\E_k\|w_{k+1} - w^*\|] &\leq \E\left[ \frac{M}{2\bar{\mu}}\|w_k - w^*\|^2 + \frac{\sigma
			\|w_k - w^*\|}{\bar{\mu}\sqrt{|S_k|}} + \frac{v}{\bar{\mu} \sqrt{|X_k|}}\right] \\
	&\leq \frac{\gamma M}{2\bar{\mu}}(\E [\|w_k - w^*\|])^2 + \frac{\sigma 
			\E[ \|w_k - w^*\|]}{\bar{\mu}\sqrt{|S_k|}} + \frac{v}{\bar{\mu} \sqrt{|X_k|}}\\
	& \leq \frac{\bar{\mu}}{3\gamma M}\tau_k \left(\frac{\tau_k}{6} \right) + 
	\frac{\bar{\mu}}{3\gamma M}\tau_k\left(\frac{1}{4}\sqrt{\frac{|S_0|}{|S_k|}}\right) 
	+ \frac{\bar{\mu}}{3\gamma M}\left(\frac{1}{2\sqrt{\eta^k_k}} \right) \\
	&= \frac{\bar{\mu}}{3\gamma M}\tau_k\left[\frac{\tau_k}{6} + \frac{1}{4}\sqrt{\frac{|S_0|}{|S_k|}} 
	+ \frac{1}{2\tau_k\sqrt{\eta^k_k}} \right] \\
	&\leq \frac{\bar{\mu}}{3\gamma M}\tau_k\left[\frac{\tau_k}{6} + \frac{1}{4}\sqrt{\frac{|S_0|}{|S_k|}} 
	+ \frac{1}{2\eta^{k/4}_k} \right] \\                                        
	&= \frac{\bar{\mu}}{3\gamma M}\tau_k\rho_k \
	\leq \ \frac{\bar{\mu}}{3\gamma M}\tau_{k+1} ,
	\end{align*}
	which proves \eqref{rlinear}.  
	
	To prove R-superlinear convergence, we show that the sequence $\tau_k$ converges superlinearly to 0. First, we use 
	induction to show that $\tau_k < 1$. For the base case $k=1$ we have,
	\begin{equation*}
	\rho_0 = \frac{\tau_0}{6} + \frac{1}{4}+ \frac{1}{2} = \frac{1}{6} + \frac{1}{4} + \frac{1}{2} = \frac{11}{12} < 1 ,
	\end{equation*}
	\begin{equation*}
	\tau_1= \max\left\{\tau_0\rho_0, \eta_1^{(-1/4)}\right\} = \max\left\{\rho_0, \eta_1^{-1/4}\right\} < 1  . 
	\end{equation*}
	Now, let us assume that $\tau_{k} < 1$ for some $k>1$. By the fact that $\{ |S_k|\}$ and $\{ \eta_k \}$ are increasing sequences, we obtain
	\begin{equation}   \label{newarg}
	\rho_k = \frac{\tau_k}{6} + \frac{1}{4}\sqrt{\frac{|S_0|}{|S_k|}} + \frac{1}{2\eta^{k/4}_k}
	\leq \frac{1}{6} + \frac{1}{4} + \frac{1}{2} = \frac{11}{12} < 1    ,
	\end{equation}
	\begin{equation*}
	\tau_{k+1} = \max\left\{\tau_k\rho_k, \eta_{k+1}^{-(k+1)/4} \right\} \leq \max\left\{\rho_k, \eta_1^{-(k+1)/4}\right\} < 1,
	\end{equation*}
	which proves that $\tau_k <1$ for all $k >1$. 
	
	Moreover,  since $\rho_k \leq 11/12$, we see from the first definition in \eqref{lots}  (and the fact that $\eta_k \rightarrow \infty$) that the sequence $\{\tau_k\}$ converges to zero. This  implies by the second definition in \eqref{lots} that $\{ \rho_k \} \rightarrow 0$.

	\comment{
		Using this fact, we have that
		\begin{align*}
		\lim\limits_{k\rightarrow \infty} \rho_k &= \lim\limits_{k\rightarrow \infty} \frac{\tau_k}{6} + \frac{1}{4}\sqrt{\frac{|S_0|}{|S_k|}} + \frac{1}{2\eta^{k/4}_k} \\
		&= \lim\limits_{k\rightarrow \infty} \frac{\max\left\{\tau_{k-1}\rho_{k-1}, \eta_k^{-k/4}\right\}}{6} +  \frac{1}{4}\sqrt{\frac{|S_0|}{|S_k|}} + \frac{1}{2\eta_k^{k/4}} \\
		& \leq \lim\limits_{k\rightarrow \infty} \frac{\max\left\{\tau_{k-1}\rho_{k-1}, 4^{-k/4}\right\})}{6} \\
		& \leq \lim\limits_{k\rightarrow \infty} \frac{\max\left\{\rho_{k-1}, 4^{-k/4}\right\}}{6} \\
		& \leq \lim\limits_{k\rightarrow \infty} \frac{\rho_{k-1}}{6}   .
		\end{align*}
		Since \textcolor{blue}{$0 \leq \rho_k <1$}, this implies that $\lim\limits_{k\rightarrow \infty} \rho_k =0$.
	}
	Using these observations, we conclude that
	\begin{align*}
	\lim\limits_{k\rightarrow \infty} \frac{\tau_{k+1}}{\tau_k} &= \lim\limits_{k\rightarrow \infty} \frac{\max\left\{\tau_k\rho_k, \eta_{k+1}^{-(k+1)/4}\right\}}{\tau_k} \\
	&= \lim\limits_{k\rightarrow \infty} \max\left\{\rho_k, \frac{\eta_{k+1}^{-(k+1)/4}}{\tau_k}\right\} \\
	& \leq \lim\limits_{k\rightarrow \infty} \max\left\{\rho_k, \left(\frac{\eta_k}{\eta_{k+1}}\right)^{k/4}\frac{1}{\eta_{k+1}^{1/4}}\right\} \\ 
	& \leq \lim\limits_{k\rightarrow \infty} \max\left\{\rho_k, \frac{1}{\eta_{k+1}^{1/4}}\right\} = 0      .
	\end{align*}      
\end{proof}  

The constants 6 and 4 in assumptions (i)-(ii) of this theorem were chosen for the sake of simplicity, to avoid introducing general parameters, and other values could be chosen. 

Let us compare this result with those established in the literature. We established superlinear convergence in expectation. In contrast the results in \cite{roosta2016sub} show a rate of decrease in the error at a given iteration with certain probability $1-\delta$. Concatenating such statements does not give convergence guarantees of the overall sequence with high probability. We should note, however, that to establish our result we assume that all iterates satisfy Assumption~II.
Pasupathy et al. \cite{2014pasglyetal} use a different approach to show that the entire sequence of iterates converge in expectation. They assume that the ``deterministic analog" has a \emph{global superlinear rate} of convergence, a rather strong assumption.  In conclusion, all the results just mentioned are useful and shed light into the properties of subsampled Newton methods, but none of them seem to be definitive.

\section{Inexact Newton-CG Method}    \label{inexactn}

We now consider inexact subsampled Newton methods in which the Newton equations \eqref{lins} are solved approximately. A natural question that arises in this context is the relationship between the accuracy in the solution of \eqref{lins}, the size of the sample $S_k$,  and the rate of convergence of the method. Additional insights are obtained by analyzing the properties of specific iterative solvers for the system \eqref{lins}, and we focus here on the conjugate gradient (CG) method. We provide a complexity analysis for an inexact subsampled Newton-CG method, and in section~\ref{complex}, compare it with competing approaches. 

In this section, we assume that the Hessian is subsampled but  the gradient is not. Since computing the full (exact) gradient is more realistic when the objective function is given by the finite sum \eqref{probrr}, we assume throughout this section that the objective is $R$.
The iteration therefore has the form
\begin{equation} \label{ncg}
w_{k+1} = w_k + p_k^r
\end{equation}
where $p_k^r$ is the approximate solution  obtained after applying $r$ steps of the CG method to the $d \times d$ linear system
\begin{equation} \label{linss}
\nabla^2R_{S_k}(w_k)p = - \nabla R(w_k), \quad\mbox{with} \quad \nabla^2R_{S_k}(w_k)= \frac{1}{|S_k|}\sum_{i \in S_k} \nabla^2 F_i(w_k).
\end{equation}
We assume that the number $r$ of CG steps performed at every iteration is constant, as this facilitates our complexity analysis which is phrased in terms of $|S_k|$ and $r$. (Later on, we consider an alternative setting where the accuracy in the linear system solve is controlled by a residual test.) Methods in which the Hessian is subsampled but the gradient is not, are sometimes called \emph{semi-stochastic}, and several variants  have  been studied in  \cite{agarwal2016second,byrd2011use,pilanci2015newton,roosta2016sub}. 

A sharp analysis of the Newton-CG method \eqref{ncg}-\eqref{linss} is difficult to perform because the convergence rate of the CG method  varies at every iteration depending on the spectrum $\{ \lambda_1 \leq \lambda_2 \leq \ldots, \leq \lambda_d\}$ of the positive definite matrix $ \nabla^2R_{S_k}(w_k)$. For example, after computing $r$ steps of the CG method applied to the linear system in \eqref{linss}, the iterate $p^r$ satisfies
\begin{equation}  \label{fine}
||p^{r} - p^*||_A^2 \leq \left(\frac{\lambda_{d-r} - \lambda_1}{\lambda_{d-r} + \lambda_1}\right)^2 \| p^0 - p^*||_A^2, \quad\mbox{with} \quad
A = \nabla^2R_{S_k}(w_k).
\end{equation} 
Here $p^*$ denotes the exact solution and $\| x \|_A^2 \defeq x^T A x$.
In addition, one can show that CG will terminate in $t$ iterations, where $t$ denotes the number of distinct eigenvalues of $\nabla^2R_{S_k}(w_k)$, and also shows that the method does not approach the solution at a steady rate. 

Since an analysis based on the precise bound \eqref{fine} is complex, we make use of  the \emph{worst case behavior of CG} \cite{GoluvanL89} which is given by 
\begin{equation} \label{CG-Convergence}
\|p^{r} - p^*\|_A \leq 2\left(\frac{\sqrt{\kappa(A)} - 1}{\sqrt{\kappa(A)} + 1}\right)^r\|p^0 - p^*\|_A ,
\end{equation}
where $\kappa(A)$ denotes the condition number of $A$.
Using this bound, we can establish the following linear-quadratic bound. Here the sample $S_k$ is allowed to vary at each iteration but its size, $|S_k|$ is assumed constant.

\begin{lem}\label{quadl} 
	%
	Let $\{w_k\}$ be the iterates generated by the inexact Newton method  \eqref{ncg}-\eqref{linss}, where $|S_k|=\beta$ and the direction $p_k^r$  is the result of performing $r <d $ CG iterations on the system  \eqref{linss}. Suppose Assumptions~A1, A3 and A4 hold. 
	Then, 
	\begin{equation}  \label{capsule}
	\E_k[\|w_{k+1} - w^*\|] \leq C_1 \|w_k - w^*\|^2 +\left( \frac{C_2}{\sqrt{\beta}} +C_3 \theta^r \right) \|w_k - w^*\| ,
	\end{equation}
	where    
	\begin{equation}   \label{c3}
	C_1 = \frac{M}{2\mu_{\beta}}, \quad C_2 =  \frac{\sigma}{\mu_{\beta}}, \quad   
	C_3 = \frac{2 L}{\mu_{\beta}}\sqrt{\delta(\beta)}, \quad
	 \theta= \left(\frac{\sqrt{\delta(\beta)}-1}{\sqrt{\delta(\beta)}+1}\right), \quad
	\delta(\beta)=  \frac{L_\beta}{\mu_\beta}.
	\end{equation}
\end{lem}
\begin{proof}  
	We have that
	\begin{align}   
	\E_k[\|w_{k+1} - w^*\|] &= \E_k[\|w_k -w^* + p_k^r\|]   \nonumber \\
	& \leq \underbrace{\E_k[\|w_k - w^* - \nabla^2R_{S_k}^{-1}(w_k)\nabla R(w_k)\|]}_\text{Term 4} +
	\underbrace{\E_k[\|p_k^r + \nabla^2R_{S_k}^{-1}(w_k)\nabla R(w_k)\|]}_\text{Term 5} . \label{twott}
	\end{align} 
	Term 4 was analyzed in the previous section where the objective function is $F$, i.e., where the iteration is defined by \eqref{iterate} so that \eqref{bterm} holds. In our setting, we have that Term~3 in \eqref{terms} is zero (since the gradient is not sampled) and hence, from \eqref{linquad},
	\begin{equation}
	\E_k[\|w_k - w^* - \nabla^2R_{S_k}^{-1}(w_k)\nabla R(w_k)\|] \leq  \frac{M}{2\mu_{\beta}}\|w_k - w^*\|^2 + \frac{1}{\mu_{\beta}}\E_k\left[\left\|\left(\nabla^2 R_{S_k}(w_k) - \nabla^2 R(w_k)\right)(w_k - w^*)\right\|\right] .
	\end{equation}
	Recalling Lemma~\ref{accuh} (with $R$ replacing $F$) and the definitions \eqref{c3}, we have 
	\begin{equation}
	\E_k[\|w_k - w^* - \nabla^2R_{S_k}^{-1}(w_k)\nabla R(w_k)\|] \leq C_1 \|w_k - w^*\|^2 + \frac{C_2}{\sqrt{\beta}} \|w_k - w^*\| .
	\end{equation}
	Now, we analyze Term 5, which is the residual in the CG solution after $r$ iterations. Assuming for simplicity that the initial CG iterate is $p^0_k=0$, we obtain from \eqref{CG-Convergence} 
	\begin{align*}
	\|p_k^r + \nabla^2R_{S_k}^{-1}(w_k)\nabla R(w_k)\|_{A} &\leq 2\left(\frac{\sqrt{\kappa(A)}-1}
	{\sqrt{\kappa(A)}+1}\right)^{r} \|\nabla^2R_{S_k}^{-1}(w_k)\nabla R(w_k)\|_{A} ,
	\end{align*}
	where $A= \nabla^2R_{S_k}(w_k)$.
	To express this in terms of unweighted norms,  note that if $\| a \|^2_{A}  \leq \| b \|^2_{A} $, then
	\[  
	\lambda_1 \|a\|^2 \leq   a^T {A} a \leq   b^T {A} b \leq \lambda_d \|b\|^2 \ \Longrightarrow \ \|a\| \leq \sqrt{k(A)} \|b\|.
	\]
	Therefore, from Assumption A1
	\begin{align}
	\E_k[ \| p_k^r + \nabla^2R_{S_k}^{-1}(w_k)\nabla R(w_k) \|] &\leq  2\sqrt{\kappa(A)}\left(\frac{\sqrt{\kappa(A)}-1}
	{\sqrt{\kappa(A)}+1}\right)^{r} \E_k [\| \nabla^2R_{S_k}^{-1}(w_k)\nabla R(w_k) \|] \nonumber \\
	&\leq 2\sqrt{\kappa(A)}\left(\frac{\sqrt{\kappa(A)}-1}{\sqrt{\kappa(A)}+1}\right)^{r} \| \nabla R(w_k)\| \, \E_k[\|\nabla^2R_{S_k}
	^{-1}(w_k)\|] \nonumber \\
	&\leq \frac{2 L}{\mu_{\beta}}\sqrt{\kappa(A)}\left(\frac{\sqrt{\kappa(A)}-1}{\sqrt{\kappa(A)}+1}\right)^{r} \|w_k - w^*\| 
	\nonumber \\
	& = C_3 \theta^r \|w_k - w^*\| ,   \label{eder}
	\end{align}  
	where the last step follows from the fact that, by the definition of $A$, we have 
	$
	     \mu_\beta \leq \lambda_1 \leq \cdots \leq \lambda_d \leq L_\beta,
	 $
	 and hence $\kappa(A) \leq L_\beta/\mu_\beta = \delta(\beta)$.
\end{proof}
We now use  Lemma~\ref{quadl} to determine the number of Hessian samples $|S|$ and the number of CG iterations $r$ that guarantee a \emph{given}  rate of convergence. Specifically, we require a linear rate with constant 1/2, in a neighborhood of the solution.  This will allow us to compare our results with those in  \cite{agarwal2016second}.  We recall that  $C_1$ is defined in \eqref{c3} and $\gamma$ in \eqref{b1}.

\begin{thm} \label{thm-cg}
Suppose that Assumptions A1, A3, A4 and B1 hold.
	Let $\{w_k\}$ be the iterates generated by  inexact Newton-CG method \eqref{ncg}-\eqref{linss}, with 
		\begin{equation}   \label{sampleb}
		|S_k|=\beta \geq \tfrac{64 \sigma^2 }{\bar \mu^2} ,
		\end{equation}
	and suppose that the number of CG steps performed at every iteration satisfies
	\begin{equation*}
	r \geq \log \left(\frac{16L}{\mu_{\beta}} \sqrt{\delta(\beta)}  \right) \frac{1}
	{  \log \left( \frac{\sqrt{\delta(\beta)} + 1}{\sqrt{\delta(\beta)} - 1} \right)  }.
	\end{equation*}
	 Then, if $||w_0 - w^*|| \leq \min \{\frac{1}{4C_1}, \frac{1}{4\gamma C_1}\}$, we have
	\begin{equation}   \label{halfing}
	\E[||w_{k+1} - w^*||] \leq \tfrac{1}{2}\E[||w_k - w^*||].
	\end{equation}  
\end{thm}
\begin{proof}
		By the definition of $C_2$ given in \eqref{c3} and \eqref{sampleb}, we have that $C_2/\sqrt{\beta} \leq 1/8$. Now,
	\begin{align*}
	C_3 \theta^r & = \frac{2 L}{\mu_{\beta}}\sqrt{\delta(\beta)}\left(\frac{\sqrt{\delta(\beta)} - 1}{\sqrt{\delta(\beta)} + 1}\right)^{r} \\ \nonumber
	& \leq  \frac{2 L}{\mu_{\beta}}\sqrt{\delta(\beta)}\left(\frac{\sqrt{\delta(\beta)} - 1}{\sqrt{\delta(\beta)} + 1}\right)^{\left(\frac{\log \left(\frac{16L}{\mu_{\beta}} \sqrt{\delta(\beta)}  \right)}
		{  \log \left( \frac{\sqrt{\delta(\beta)} + 1}{\sqrt{\delta(\beta)} - 1} \right)  }\right)} \\ \nonumber
	&=\frac{2 L}{\mu_{\beta}}\sqrt{\delta(\beta)}\frac{1}{\frac{16L}{\mu_{\beta}} \sqrt{\delta(\beta)}} \quad  \leq \frac{1}{8}.
	\end{align*}
	We use induction to prove \eqref{halfing}. For the base case we have from Lemma~\ref{quadl},
	\begin{align*}
	\E [\| w_1 - w^*\|] &\leq C_1 \| w_0 - w^*\| \| w_0 - w^*\| + \left(\frac{C_2}{\sqrt{\beta}} + C_3 \theta^r\right) \|w_0 - w^*\| \\
	& \leq \tfrac{1}{4} \|w_0 - w^*\| + \tfrac{1}{4}\|w_0 - w^*\| \\
	&= \tfrac{1}{2} \|w_0 - w^*\| .
	\end{align*}
	Now suppose that \eqref{halfing} is true for $k^{th}$ iteration. Then, 
	\begin{align*}
	\E [\E_k\|w_{k+1} - w^*\|] &\leq C_1\E[\|w_k - w^*\|^2]+  \left(\frac{C_2}{\sqrt{\beta}} + C_3 \theta^r\right) \E[\|w_k - w^*\|]\\
	&\leq \gamma C_1\E[\|w_k - w^*\|] \, \E[\|w_k - w^*\|] +  \left(\frac{C_2}{\sqrt{\beta}} + C_3 \theta^r\right) \E[\|w_k - w^*\|]\\
	& \leq \tfrac{1}{4}\E [\|w_k - w^*\|] + \tfrac{1}{4}\E[\|w_k - w^*\|] \\
	& \leq \tfrac{1}{2}\E[\|w_k - w^*\|] .
	\end{align*}
\end{proof}

We note that condition (\ref{sampleb}) may require a value for $\beta$ greater than $N$. In this case the proof of Theorem \ref{thm-cg} is clearly still valid if we sample with replacement, but this is a wasteful strategy since it achieves the bound $|S_k| > N$ by repeating samples. If we wish to sample without replacement in this case, we can set $\beta =N$. Then our Hessian approximation is exact and the $C_2$ term is zero, so the proof still goes through and Theorem \ref{thm-cg} holds.

This result was established using the worst case complexity bound of the CG method. We know, however, that  CG converges to the solution in at most $d$ steps.  Hence, the bound on the maximum number of iterations needed to obtain a linear rate of convergence with constant 1/2 is
\begin{equation} \label{min-r}
r = \min \left\{d, \frac{\log \left(16L\sqrt{\delta(\beta)}/\mu_{\beta}\right) }{\log\left(\frac{\sqrt{\delta(\beta) }+ 1}{\sqrt{\delta(\beta)} - 1}\right)} \right\} .
\end{equation} 

\medskip  					
\paragraph{Convergence Rate Controlled by  Residual Test.}
In many practical implementations of inexact Newton methods, the CG iteration is stopped based on the norm of the residual in the solution of the linear system \eqref{lins}, rather than on a prescribed maximum number $r$ of CG iterations \cite{DembEiseStei82}. For the system \eqref{linss}, this residual-based termination test is given by
\begin{equation}
\|\nabla^2 R_{S_k}(w_k)p_k^r - \nabla R(w_k)\| \leq \zeta\|\nabla R(w_k)\| ,
\end{equation}
where $\zeta <1$ is a control parameter. Lemma~\ref{quadl} still applies in this setting, but with a different constant $C_3  \theta^r$. Specifically, Term~5 in \eqref{twott} is modified as follows:
\begin{align*}
\E_k[\|p_k^r + \nabla^2 R_{S_k}^{-1}(w_k)\nabla R(w_k)\|] &\leq \E_k[\|\nabla^2 R_{S_k}^{-1}(w_k)\|\|\nabla^2 R_{S_k}(w_k)p_k^r - \nabla R(w_k)\|] \\
& \leq \frac{\zeta}{\mu_{\beta}} \|\nabla R(w_k)\| \\
& \leq \frac{L \zeta}{\mu_{\beta}}\|w_k -w^*\|.
\end{align*}
Hence, comparing with \eqref{eder},  we now have that $C_3  \theta^r = \frac{L}{\mu_{\beta}}\zeta$.
To obtain linear convergence with constant 1/2, we must impose a bound on the parameter  $ \zeta$, so as to match the analysis in 
Theorem~\ref{thm-cg}, where we required that $C_3  \theta^r \leq \frac{1}{8}$. This condition is satisfied if  
\begin{equation*}
\zeta \leq \frac{\mu_{\beta}}{8L}   .
\end{equation*}
Thus, the parameter $\zeta$ must be inversely proportional to a quantity related to the condition number of the Hessian.  


We conclude this section by remarking that the results presented in this section may not reflect the full power of the subsampled Newton-CG method since we  assumed the worst case behavior of CG, and as noted in \eqref{fine}, the per-iteration progress can be much faster than in the worst case.

     \section{ Comparison with Other Methods}  \label{complex}
     We now ask whether the CG method is, in fact, an efficient linear solver when employed in the  inexact subsampled Newton-CG method \eqref{ncg}-\eqref{linss},  or whether some other iterative linear solver could be preferable. Specifically, we compare CG with a semi-stochastic gradient iteration (SGI) that is described below; we denote the variant of \eqref{ncg}-\eqref{linss} that uses the SGI iteration to solve linear systems as the Newton-SGI method. Following  Agarwal et al. \cite{agarwal2016second}, we measure efficiency by estimating the total amount of Hessian-vector products required to achieve  a local linear rate of convergence with convergence constant 1/2 (the other costs of the algorithms are identical).   
     
To present the complexity results of this section we introduce the following definitions of condition numbers:      
 
      \begin{equation*}
      \hat{\kappa} = \frac{\bar{L}}{\mu}\, , \quad \text{   }\\ \hat{\kappa}^{\max} = \frac{\bar{L}}{\bar{\mu}}\,  \quad \text{  and }\\  \kappa = \frac{L}{\mu}.
      \end{equation*}
     
     \paragraph{Newton-CG Method.}
     Since each CG iteration requires 1 Hessian-vector product, every iteration of the inexact subsampled Newton-CG method requires $\beta \times r$ Hessian-vector products, where $\beta=|S_k|$ and $r$ is the number of CG iterations performed.
     
By the definitions in Assumptions I, we have that $\sigma^2$ is bounded by a multiple of $\bar{L}^2$. Therefore, recalling the definition \eqref{c3} we have that the sample size stipulated in Theorem~\ref{thm-cg} and \eqref{c3} satisfies
     \begin{align*}
     \beta = O(C_2^2)=O(\sigma^2 /\bar{\mu}^2)=O((\hat{\kappa}^{\max})^2).
     \end{align*}
Now, from \eqref{min-r} the number of CG iterations satisfies the bound  
     \begin{align*}
     r &=O\left(\min \left\{d, \frac{\log((L/\mu_{\beta})\sqrt{\delta(\beta)})}{\log\left(\frac{\sqrt{\delta{(\beta)}} + 1}{\sqrt{\delta(\beta)} - 1}\right)}\right\}\right)  \\
     &=O\left(\min \bigg \{d, \sqrt{\hat{\kappa}^{\max}}\log(\hat{\kappa}^{\max})\bigg\}\right) ,
     \end{align*}
     where the last equality is by the fact that $\delta(\beta) \leq \hat{\kappa}^{\max}$ and $L \leq \bar{L}$.
     Therefore, the number of Hessian-vector products required by the Newton-CG method to yield a linear convergence rate with constant of 1/2 is
     \begin{align}    \label{reversal}
     O(\beta r)= O\left((\hat{\kappa}^{\max})^2\min \bigg\{d, \sqrt{\hat{\kappa}^{\max}}\log(\hat{\kappa}^{\max})\bigg\}\right) .
     \end{align}

     \paragraph{Newton-SGI Method.} To motivate this method, we first note that a step of the classical Newton method is given by the minimizer of the quadratic model
      \begin{equation}   \label{quadratic}
     Q(p) = R(w_k) + \nabla R(w_k)^Tp + \frac{1}{2}p^T\nabla^2 R(w_k) p.
     \end{equation}
     We could instead minimize $Q$ using the gradient method,
          \[
     p_{t+1} = p_t - \nabla Q(p_t) = (I - \nabla^2 R(w_k))p_t  - \nabla R(w_k), 
     \]
but the cost of the Hessian-vector product in this iteration is expensive. Therefore, one can consider  the semi-stochastic gradient iteration (SGI)
           \begin{equation} \label{sgi}
     p_{t+1} = 
     (I - \nabla^2 R_i(w_k))p_t  - \nabla R(w_k), 
     \end{equation}  
     where the index $i$ is chosen at random from $\{1, \ldots, N\}$.  
     We define the Newton-SGI  method by $w_{k+1}= w_k
 +p_r$, where $p_r$ is the iterate obtained after applying $r$ iterations of \eqref{sgi}.
     
 Agarwal et al. \cite{agarwal2016second} analyze a method they call LiSSA that is related to this Newton-SGI method. Although they present their method as one based on a power expansion of the inverse Hessian, they note in \cite[section 4.2]{agarwal2016second} that, if the outer loop in their method is disabled (by setting $S_1=1$), then their method is equivalent to our Newton-SGI method. They provide a complexity bound for the more general version of the method in which  they compute $S_2$ iterations of \eqref{sgi}, repeat this $S_1$ times, and then calculate the average of all the solutions to define the new iterate. They provide a bound,  in probability, for one step of their overall method, whereas our bounds for Newton-CG are in expectation. In spite of these differences,  it is interesting to compare the complexity  of the two methods. 
     
     The number of Hessian-vector products for LiSSA (which is given  $O(S_1S_2)$ in their notation) is
     \begin{align}
     O((\hat{\kappa}^{\max})^2\hat{\kappa}\log(\hat{\kappa}) \log(d))
     \end{align}

When comparing this estimate with \eqref{reversal} we observe that the Newton-CG bounds depend on the square root of a condition number, whereas Newton-SGI depends on the condition number itself. Furthermore,  Newton-CG also has an improvement of $ \log(d)$ because our proof techniques avoid the use of matrix concentration bounds.

     \paragraph{Newton Sketch.} Let us also consider the Newton sketch algorithm recently proposed  in \cite{pilanci2015newton}.  This paper provides a linear-quadratic convergence rate \cite[Theorem 1]{pilanci2015newton} similar in form to \eqref{threet}, but stated in probability. In order to achieve local linear convergence, their user-supplied parameter $\epsilon$ should be of the order $1/\kappa$. Therefore the sketch dimension $m$ is given by
     \begin{align*}
     m=O\left(\frac{W^2}{\epsilon^2}\right)= O(\kappa^2 \min\{n,d\}),
     \end{align*}
     where the second equality follows from the fact that the square of the Gaussian width $W$ is at most $\min\{n,d\}$.
     The per iteration cost of Newton sketch with a randomized Hadamard transform is
     \begin{equation*}
     O(nd\log(m) + dm^2) = O(nd\log(\kappa d) + d^3\kappa^4),
     \end{equation*}   
     and the algorithm is designed for the case  $n>>d$. 
     
     We note in passing that certain implicit assumptions are made about the algorithms discussed above when the objective is given by the finite sum $R$. In subsampled Newton methods, it is assumed that the number of subsamples is less than the number of examples $n$; in Newton sketch,  the sketch dimension is assumed to be less than $n$. 
     This implies that for all these methods one makes the implicit assumption that $n> \kappa^2$.   
%
     We should also note that in all the stochastic second order methods, the number of samples or sketch dimension required by the theory is $\kappa^2$, but in practice a small number of samples or sketch dimensions suffice to give good performance. 
This suggests that the theory could be improved and that techniques other than concentration bounds might help in achieving this. 
   
      \paragraph{Work complexity to obtain an $\epsilon$-accurate solution.}   
 Table 4.1  compares  a variety of methods in terms of the total number of gradient and Hessian-vector products required to obtain an $\epsilon$-accurate solution. The results need to be interpreted with caution as the convergence rate of the underlying methods differs in nature, as we explain below. Therefore, Table~4.1 should be regarded mainly as summary of results in the literature and not as a simple way to rank methods. In stating these results, we assume that  the cost of a Hessian-vector product is same as the cost of a gradient evaluation, which is realistic in many (but not all) applications. 
 
      \begin{table}[htp]
     	\centering
      	\caption{Time complexity to obtain an $\epsilon$ accurate solution. Comparison of the Newton-CG(Inexact) method analyzed in this paper with other well known methods. The third column reports orders of magnitude.}
     	\label{complexity}
     	\setlength{\tabcolsep}{10pt} 
     	\renewcommand{\arraystretch}{1.5}
     	\begin{tabular}{|c|c|c|c|}
     		\hline
     		Method                         & Convergence & Time to reach $\epsilon$-accurate solution. & Reference           \\ \hline
     		SG                             & Global         & $\frac{d\omega \kappa^2}{\epsilon}$  &      \cite{BottLeCu04}                                          \\ 
     		DSS                             & Global         & $\frac{dv \kappa}{\mu\epsilon}$   & \cite{byrd2012sample}                                     \\
     		GD                             & Global         & $nd\kappa \log(\frac{1}{\epsilon})$  &  \cite{mybook}                                      \\ 
     		Newton                      & Local          & $nd^2\log\log(\frac{1}{\epsilon})$  &  \cite{mybook}                                \\ 
     		Newton-CG(Exact)          & Local          & $(n + (\hat\kappa^{\max})^2d)d\log(\frac{1}{\epsilon})$  & [{\small This paper}]                 \\ 
     		{\bf Newton-CG(Inexact)} & Local          & $(n + (\hat\kappa^{\max})^2\sqrt{\hat\kappa^{\max}})d\log(\frac{1}{\epsilon})$ & [{\small This paper}]                                      \\ 
     		LiSSA                          & Local          & $(n + (\hat\kappa^{\max})^2\hat{\kappa})d\log(\frac{1}{\epsilon})$ & \cite{agarwal2016second}                                             \\ 
     		Newton Sketch               & Local    & $(n + \kappa^4d^2)d\log(\frac{1}{\epsilon})$ &  \cite{pilanci2015newton}                          \\ 
     		
  \hline
     	\end{tabular}
     \end{table}

 In Table 4.1, SG is the classical stochastic gradient method with diminishing step-sizes. The complexity results of SG do not depend on $n$ but depend on $\kappa^2$, and are inversely proportional to $\epsilon$ due to its sub-linear rate of convergence. The constant $\omega$ is the trace of the inverse Hessian times a covariance matrix; see \cite{BottouBosq08}.  DSS is subsampled gradient method where the Hessian is the identity (i.e., no Hessian subsampling is performed) and the gradient sample size $|X_k|$ increases at a geometric rate. The complexity bounds for this method are also independent of $n$, and depend on $\kappa$ rather than $\kappa^2$ as in SG.

GD  and  Newton are the classical deterministic gradient descent and Newton methods. 
   Newton-CG (Exact and Inexact) are the subsampled Newton methods  discussed in this paper. In these methods, $(\hat{\kappa}^{\max})^2$ samples are used for Hessian sampling, and the number of inner CG iterations is of order $O(d)$ for the exact method, and $O(\sqrt{\hat{\kappa}^{\max}}$)  for the inexact method. LiSSA is the method proposed in \cite{agarwal2016second}, wherein the inner solver is a semi-stochastic gradient iteration; i.e., it is similar to our Newton-SGI method but we quote the complexity results from \cite{agarwal2016second}.  The bounds for LiSSA differ from those of Newton-CG in a square root of a condition number. Newton-Sketch has slightly higher cost compared to Newton-CG and LiSSA.  
   
    
{\em  A note of caution:} Table 4.1 lists methods with different types of convergence results.  For GD and Newton, convergence is deterministic; for SG, DSS and Newton-CG(Exact \& Inexact), convergence is in expectation; and for LiSSA and Newton Sketch the error (for a given iteration) is in probability. The definition of an $\epsilon$-accurate solution also varies. For all the first order methods (SG, DSS, GD)  it represents accuracy in function values;  for all the second-order methods (Newton, Newton-CG, LiSSA, NewtonSketch) it represents accuracy in the iterates ($\|w-w^*\|$). Although for a strongly convex function, these two measures are related, they involve a different constant in the complexity estimates.

\section{Numerical Experiments}   \label{numerical}
\setcounter{table}{0}
We conducted numerical experiments to illustrate the performance of the inexact subsampled Newton methods discussed in section~\ref{inexactn}. We consider binary classification problems where the training objective function is given by the logistic loss  with $\ell_2$ regularization:
\begin{equation} \label{logistic-loss}
R(w)=\frac{1}{N} \sum_{i=1}^{N}\log(1 + \exp(-y^iw^Tx^i)) + \frac{\lambda}{2}\|w\|^2 .
\end{equation} 
The regularization parameter  is chosen as $ \lambda = \frac{1}{N}$. 
The iterative linear solvers, CG and SGI, require Hessian-vector products, which are computed through the formula
\begin{align*}
\nabla^2 R_{s_k}(w) p = \frac{1}{|S_k|}\sum_{i \in S_k}\frac{\exp(-y^iw^Tx^i)}{(1 + \exp(-y^iw^Tx^i))^2}x^i((x^i)^Tp) + \lambda p .
\end{align*} 

Table~5.1 summarizes the datasets used for the experiments.  Some of these datasets divide the data into training and testing sets; for the rest, we randomly divide the data so that the training set constitutes  70\% of the total. In Table~5.1, $N$ denotes the total number of examples in a data set, including  training and testing points.

\begin{table}[htp]
	\centering
	\label{data sets}
	\begin{tabular}{|c||c|c|c|}
		\hline
		Data Set  & Data Points $N$ & Variables $d$ & Reference \\ \hline
		MNIST    & 70000         & 784         &      \cite{lecun2010mnist}     \\ 
		Covertype & 581012       & 54          &           \cite{blackard1999comparative}\\ 
		Mushroom & 8124          & 112         &          \cite{Lichman2013} \\ 
		Synthetic & 10000          & 50          &           \cite{mukherjee2013parallel}\\ 
		CINA      & 16033         & 132         &          \cite{CINA} \\ 
		Gisette   & 7000		   & 5000		 &			\cite{guyon2004result} \\ \hline
	\end{tabular}
	\caption{A description of binary datasets used in the experiments}
\end{table}
\newpage
The following methods were tested in our experiments. 
\begin{itemize}
	\item[] {\bf GD}. The gradient descent method  $w_{k+1}= w_k - \alpha_k \nabla R(w_k)$.
	\item[] {\bf Newton}. The exact Newton method $w_{k+1}= w_k + \alpha_k p_k$, where $p_k$ is the solution of the system 
	$	      \nabla^2R(w_k)p_k = - \nabla R(w_k)$
	computed to high accuracy by the conjugate gradient method.
	\item[ ] {\bf Newton-CG}. The inexact subsampled Newton-CG method $w_{k+1}= w_k + \alpha_k p^r_k$, where $p_k^r$ is an inexact solution of the linear system \begin{equation}  \label{repeat}
	\nabla^2R_{S_k}(w_k)p_k = - \nabla R(w_k)
	\end{equation} computed  using the conjugate gradient method. The set $S_k$ varies at each iteration but its cardinality $|S_k|$ is constant.
	\item[ ] {\bf Newton-SGI}. The inexact subsampled Newton-SGI method $w_{k+1}= w_k + \alpha_k p_k$, where $p_k$ is an inexact solution of \eqref{repeat} computed by the stochastic gradient iteration \eqref{sgi}. 
\end{itemize}
All  these methods implement an Armijo back tracking line search to determine the steplength $\alpha_k$, employ the full gradient $\nabla R(w)$,  and differ in their use of second-order information. In the Newton-CG method, the CG iteration is terminated when one of the following two conditions is satisfied:
\begin{equation}   \label{twocond}
\|\nabla^2 R_{S_k}(w_k)p_k^j - \nabla R(w_k)\| \leq \zeta \|\nabla R(w_k)\| \qquad\mbox{or} \quad j ={\rm max}_{cg},
\end{equation}
where $j$ indices the CG iterations.
The parameters in these tests were set as $\zeta=0.01$ and ${\rm max}_{cg}=10$, which are common values  in practice.  These parameter values were chosen beforehand and were not  tuned to our test set.



In all the figures below, \emph{training error} is defined as $R(w) - R(w^*)$, where $R$ is defined in terms of the data points given by the training set; \emph{testing error} is  defined as $R(w)$, without the regularization term (and using the data points from the test set). 

We begin by reporting results on the {\tt Synthetic} dataset, as they  are representative of what we have observed in our experiments. Results on the other datasets are given in the Appendix.
In Figure~5.1, we compare GD, Newton and three variants of the Newton-CG method with sample sizes $|S_k|$ given as  5\%, 10\% and 50\% of the training data. We generate two plots:  a) Training error vs. iterations; and b) Training error vs. \emph{number of effective gradient evaluations}, by which we mean that each Hessian-vector product is equated with a gradient and function evaluation. In Figure~5.2 we  plot testing error vs. time. Note that the dominant computations in these experiments are  gradient evaluations, Hessian-vector products, and  function evaluations in the line search. 

\begin{figure}[!htp]
	\label{synthetic-ex1a} 
	\includegraphics[width=1.0\linewidth]{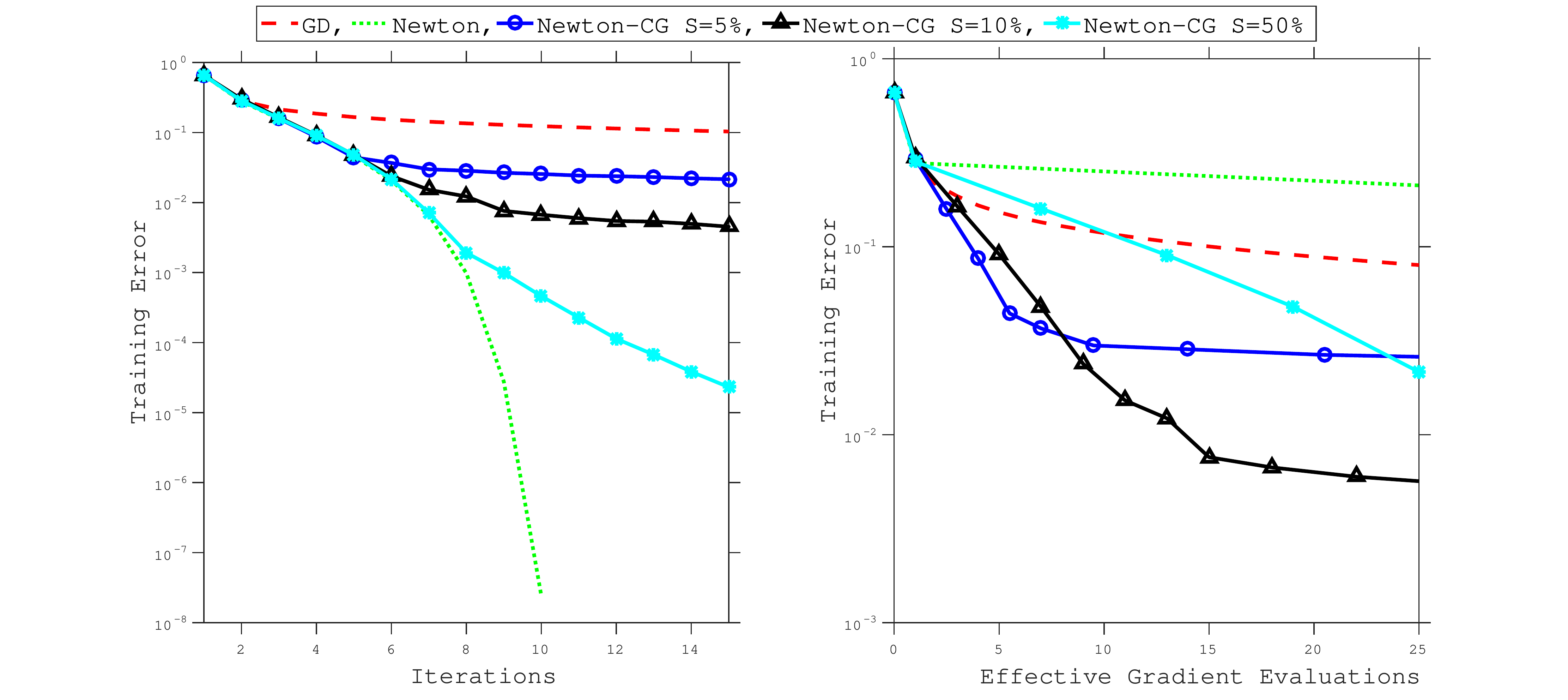}
	\caption{ {\bf Synthetic Dataset}: Performance of the inexact subsampled Newton method (Newton-CG), using three values of the sample size, and of the GD and Newton methods. Left: Training Error vs. Iterations; Right: Training Error vs. Effective Gradient Evaluations.}
\end{figure}  
\begin{figure}[!htp]
	\label{synthetic-ex1b} 
	\includegraphics[width=1.0\linewidth]{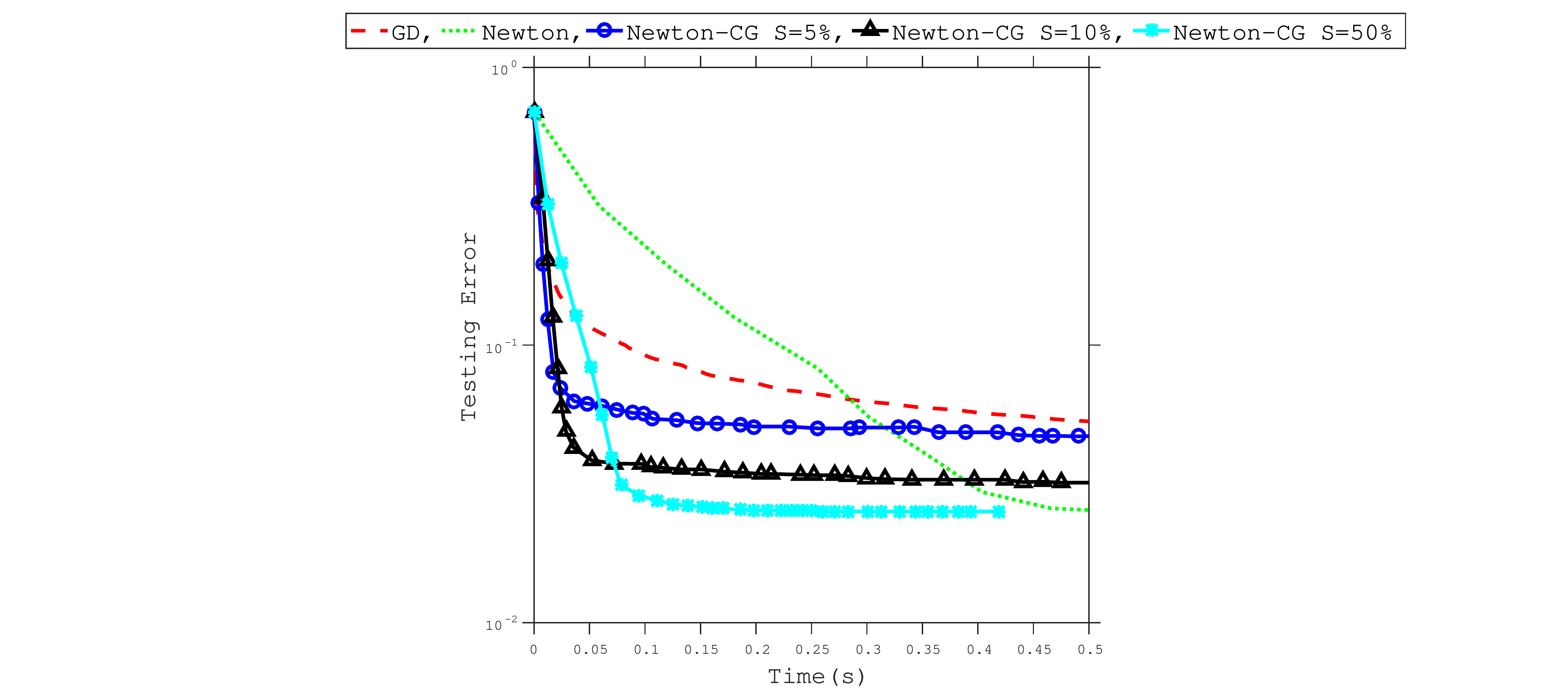}
	\caption{ {\bf Synthetic Dataset}: Comparison of the 5 methods in Figure~5.1, this time plotting Testing Error vs. CPU Time. }
\end{figure}              
\newpage         

Results comparing GD, Newton and Newton-CG on the rest of the test problems are given in the Appendix.

In the second set of experiments, reported in Figures~5.3 and 5.4, we compare  Newton-CG and Newton-SGI, again on the {\tt Synthetic} dataset. We note that Newton-SGI  is  similar to the method denoted as LiSSA in \cite{agarwal2016second}. That method contains an outer iteration that averages iterates, but in the tests reported in  \cite{agarwal2016second}, the outer loop was disabled (by setting their parameter $S_1$ to 1), giving rise to the Newton-SGI iteration.  To guarantee convergence of the SGI iteration \eqref{sgi} (which uses a unit steplength) one must ensure that the spectral norm of the Hessian for each data point is strictly less than 1; we enforced this by rescaling the data. To determine the number of  inner stochastic gradient iterations in SGI, we proceeded as follows. First, we chose \emph{one} sample size $\beta= |S|$ for  the Newton-CG method, as well as the maximum number ${\rm max}_{cg}$ of CG iterations. Then, we set the number of SGI iterations to be
$It=\beta \times {\rm max}_{cg}$, so that the per iteration number of Hessian-vector products in the two methods is similar.  We observe from Figure~5.3 that Newton-CG and Newton-SGI perform similarly in terms of effective gradient evaluations, but note from Figure~5.4 the Newton-SGI has higher computing times due to the additional communication cost involved in  individual Hessian-vector products. Similar results can be observed for the test problems in the Appendix.


\begin{figure}[!htp]
	\label{synthetic-ex2a} 
	\includegraphics[width=1\linewidth]{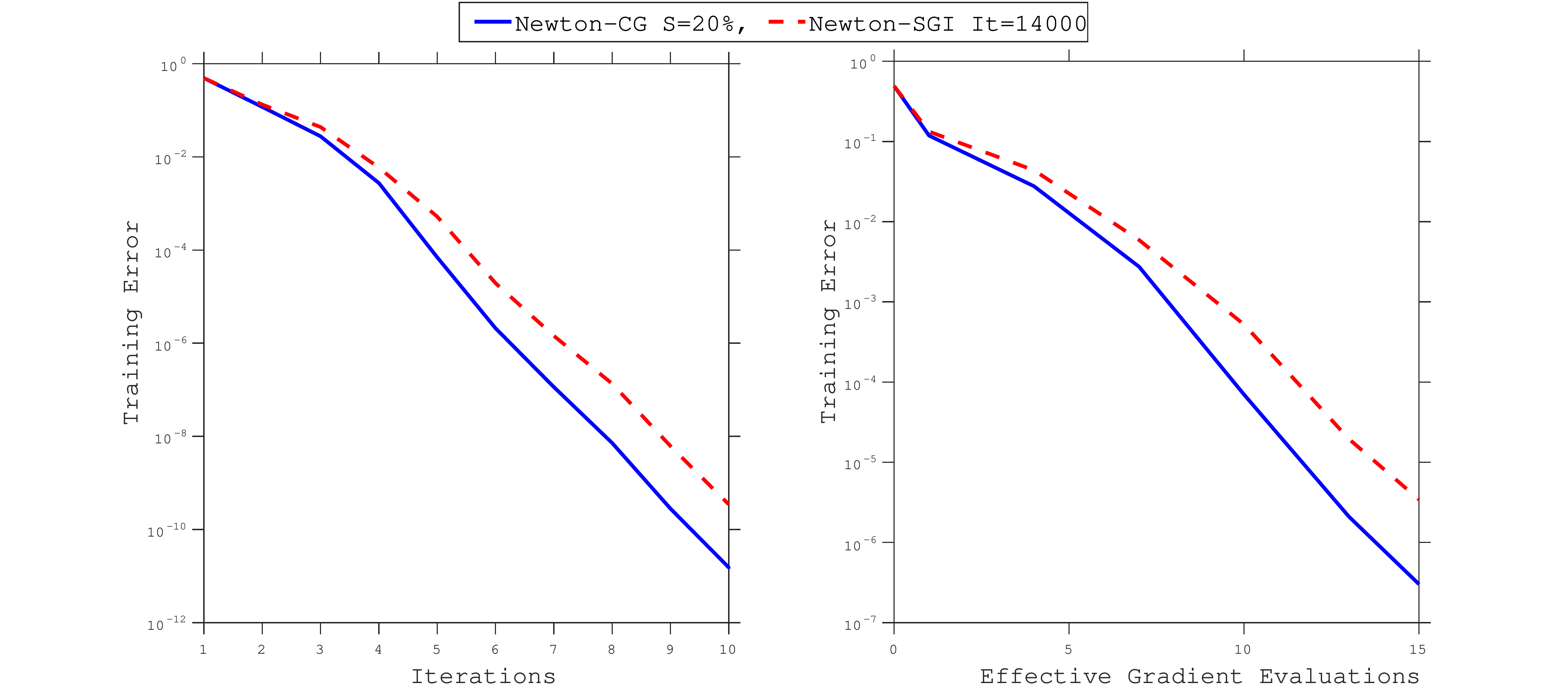}
	\caption{ {\bf Synthetic Dataset (scaled)}: Comparison of Newton-CG and Newton-SGI. Left: Training Error vs. Iterations; Right: Training Error vs. Effective Gradient Evaluations. Here {\tt It} denotes the number of iterations of  the SGI  algorithm \eqref{sgi} performed at every iteration of Newton-SGI.}
\end{figure}
\begin{figure}[!htp]
	\label{synthetic-ex2b} 
	\includegraphics[width=1\linewidth]{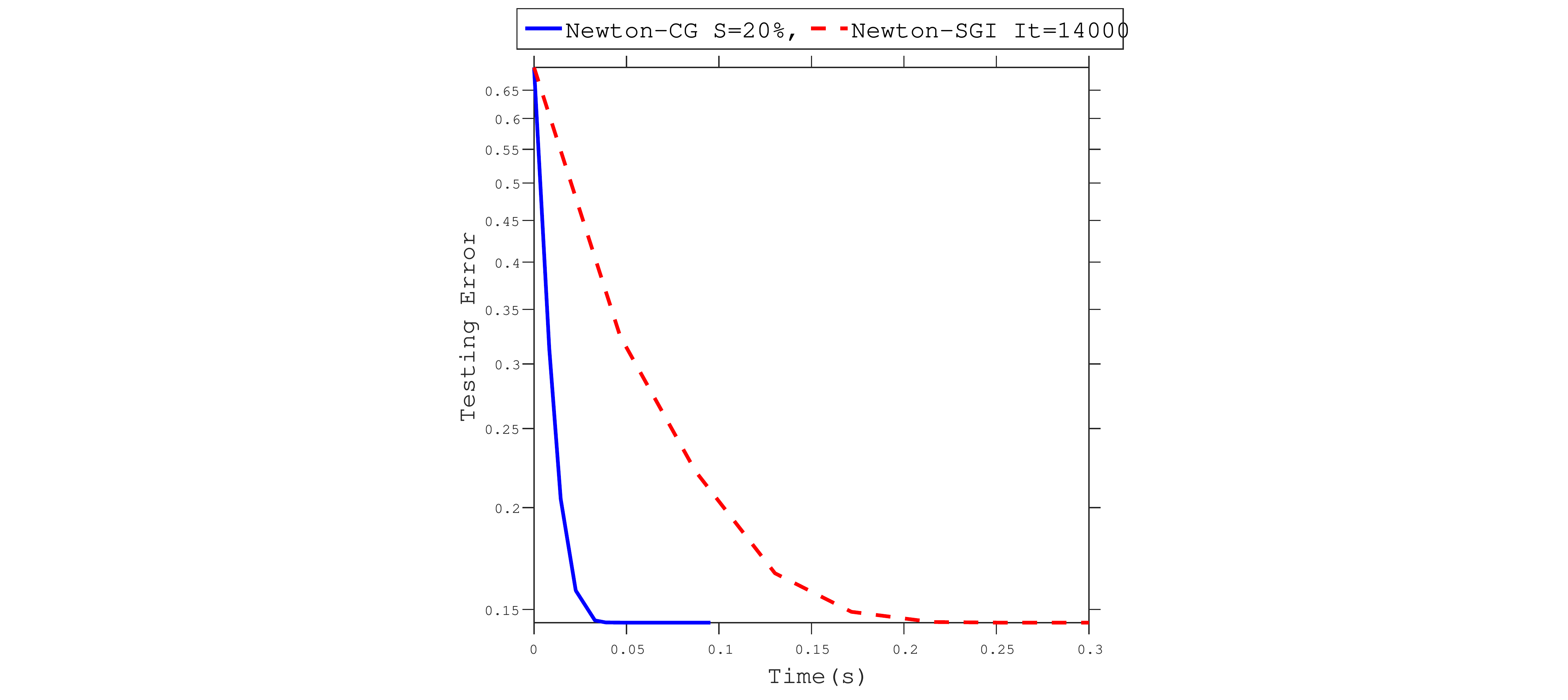}
	\caption{ {\bf Synthetic Dataset (scaled)}: Comparison of Newton-CG with Newton-SGI, this time plotting Testing Error vs. Time.}
\end{figure}           

In the third set of experiments, reported in Figures~5.5 and 5.6, we compare the Newton-CG and Newton-SGI methods on the datasets \emph{without scaling}, i.e., the spectral norm of the Hessians is now allowed to be greater than 1. To ensure convergence, we modify the SGI iteration \eqref{sgi} by incorporating a step-length parameter $\alpha_{sgi}$, yielding the following iteration:
\begin{equation} \label{sgi-unscaled}
p_{t+1} = p_t - \alpha_{sgi}\nabla Q_i(p_t) = (I - \alpha_{sgi}\nabla^2 F_i(w_k))p_t  - \alpha_{sgi}\nabla R(w_k)  .
\end{equation}    
The steplength parameter $\alpha_{sgi}$ was chosen as the value in  $\{2^{-20},\dots,2^3\}$ that gives best overall performance. 
\begin{figure}[!htp]
	\label{synthetic-ex3a} 
	\includegraphics[width=1\linewidth]{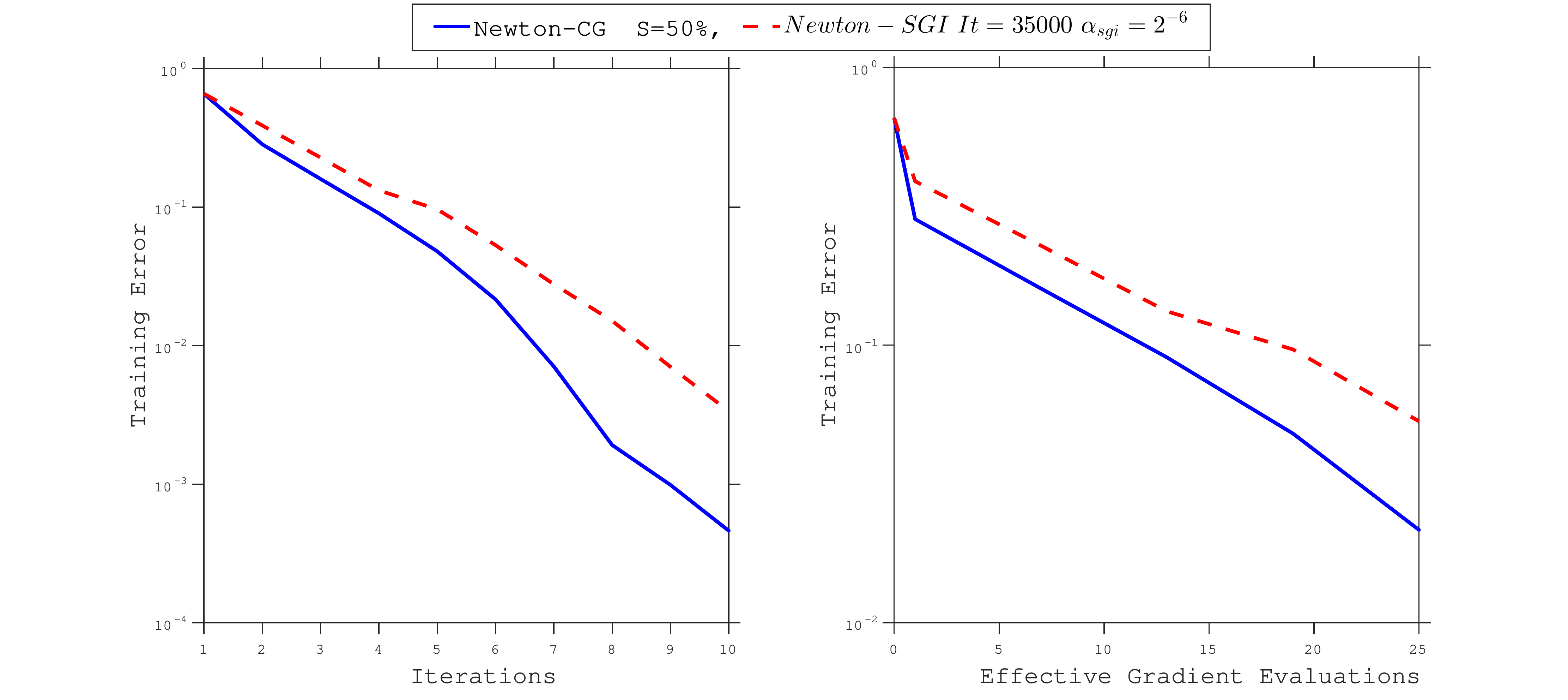}
	\caption{ {\bf Synthetic Dataset (unscaled)}: Comparison of Newton-CG with Newton-SGI. Left: Training Error vs. Iterations; Right: Training Error vs. Effective Gradient Evaluations. The parameter $\alpha_{sgi}$ refers to the steplength in \eqref{sgi-unscaled}.}
\end{figure} 
\begin{figure}[!htp]
	\label{synthetic-ex3b} 
	\includegraphics[width=1\linewidth]{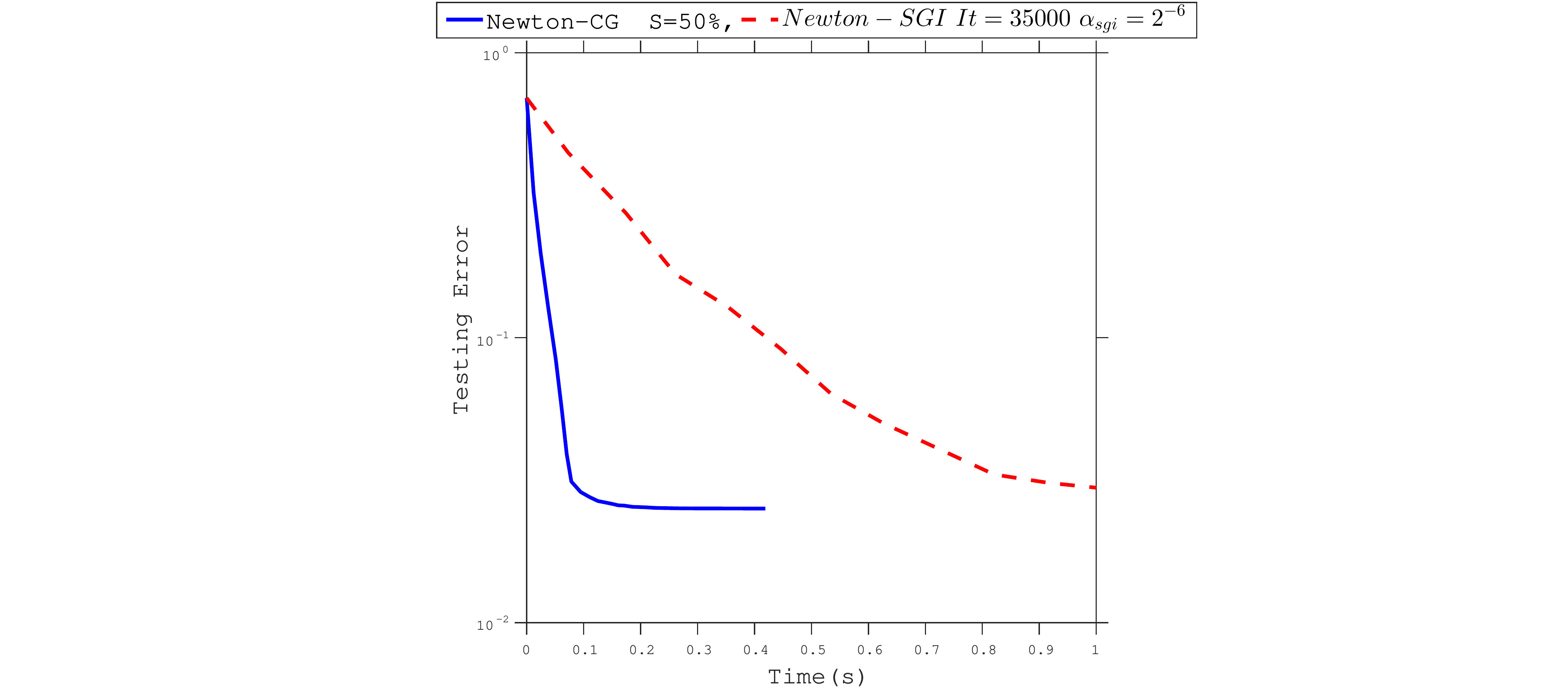}
	\caption{ {\bf Synthetic Dataset (unscaled)}: Comparison of Newton-CG with Newton-SGI, this time plottingTesting Error vs. Time. }
\end{figure}            
\newpage 

Results comparing Newton-CG and Newton-SGI on the rest of the test datasets are given in the Appendix. Overall, the numerical experiments reported in this paper suggest that the inexact subsampled Newton methods are quite effective in practice. Both versions, Newton-CG and Newton-SGI, are worthy of further investigation on a wider range of test problems. 

\section{Final Remarks}
Subsampled Newton methods  \cite{agarwal2016second,byrd2011use,byrd2012sample,erdogdu2015convergence,Martens10,roosta2016sub,xu2016sub}
are attractive in large-scale applications due to their ability to incorporate \emph{some}  second-order information at low cost. They are more stable than first-order methods and can yield a faster rate of convergence. In this paper, we  established conditions under which a method that subsamples the gradient and the Hessian enjoys a superlinear rate of convergence in expectation. To achieve this, the sample size used to estimate the gradient is increased at a rate that is faster than geometric, while the sample size for the Hessian approximation can increase at any rate.

The paper also studies the convergence properties of an inexact subsampled Newton method in which the step computation is performed by means of the conjugate gradient method. As in
\cite{agarwal2016second,erdogdu2015convergence,pilanci2015newton,roosta2016sub,xu2016sub} this
method employs the full gradient and approximates the Hessian by subsampling. We give bounds on the total amount of work needed to achieve a given linear rate of convergence, and compare these bounds with those given in \cite{agarwal2016second} for an inexact Newton method that solves linear systems using a stochastic gradient iteration. Computational work is measured by the number of evaluations of individual gradients and Hessian vector products.

Recent results on subsampled Newton methods \cite{erdogdu2015convergence,roosta2016sub,xu2016sub} establish a rate of decrease at every iteration, in probability. The results of this paper are stronger in that we establish convergence in expectation, but we note that in order to do so we introduced assumption \eqref{b1}. Recent work on subsampled Newton methods focuses on the effect of non-uniform subsampling  \cite{xu2016sub} but in this paper we consider only uniform sampling. 


The numerical results presented in this paper, although  preliminary,  make a good case for the value of subsampled Newton methods, and suggest that a more detailed and comprehensive investigation is worthwhile. We leave that study as a subject for future research. 

\bibliographystyle{plain}
\bibliography{InexactMethods}

\begin{thebibliography}{10}

\bibitem{agarwal2016second}
Naman Agarwal, Brian Bullins, and Elad Hazan.
\newblock Second order stochastic optimization in linear time.
\newblock {\em arXiv preprint arXiv:1602.03943}, 2016.

\bibitem{amaran2014simulation}
Satyajith Amaran, Nikolaos~V Sahinidis, Bikram Sharda, and Scott~J Bury.
\newblock Simulation optimization: a review of algorithms and applications.
\newblock {\em 4OR}, 12(4):301--333, 2014.

\bibitem{Bert95}
D.~P. Bertsekas.
\newblock {\em Nonlinear Programming}.
\newblock Athena Scientific, Belmont, Massachusetts, 1995.

\bibitem{blackard1999comparative}
Jock~A Blackard and Denis~J Dean.
\newblock Comparative accuracies of artificial neural networks and discriminant
  analysis in predicting forest cover types from cartographic variables.
\newblock {\em Computers and electronics in agriculture}, 24(3):131--151, 1999.

\bibitem{BottLeCu04}
L.~Bottou and Y.~Le~Cun.
\newblock {On-line Learning for Very Large Datasets}.
\newblock {\em Applied Stochastic Models in Business and Industry},
  21(2):137--151, 2005.

\bibitem{BottouBosq08}
Leon Bottou and Olivier Bousquet.
\newblock The tradeoffs of large scale learning.
\newblock In J.C. Platt, D.~Koller, Y.~Singer, and S.~Roweis, editors, {\em
  Advances in Neural Information Processing Systems 20}, pages 161--168. MIT
  Press, Cambridge, MA, 2008.

\bibitem{byrd2011use}
Richard~H Byrd, Gillian~M Chin, Will Neveitt, and Jorge Nocedal.
\newblock On the use of stochastic {H}essian information in optimization
  methods for machine learning.
\newblock {\em SIAM Journal on Optimization}, 21(3):977--995, 2011.

\bibitem{byrd2012sample}
Richard~H Byrd, Gillian~M Chin, Jorge Nocedal, and Yuchen Wu.
\newblock Sample size selection in optimization methods for machine learning.
\newblock {\em Mathematical programming}, 134(1):127--155, 2012.

\bibitem{DembEiseStei82}
R.~S. Dembo, S.~C. Eisenstat, and T.~Steihaug.
\newblock Inexact-{N}ewton methods.
\newblock 19(2):400--408, 1982.

\bibitem{erdogdu2015convergence}
Murat~A Erdogdu and Andrea Montanari.
\newblock Convergence rates of sub-sampled newton methods.
\newblock In {\em Advances in Neural Information Processing Systems}, pages
  3034--3042, 2015.

\bibitem{friedlander2012hybrid}
Michael~P Friedlander and Mark Schmidt.
\newblock Hybrid deterministic-stochastic methods for data fitting.
\newblock {\em SIAM Journal on Scientific Computing}, 34(3):A1380--A1405, 2012.

\bibitem{fu2015handbook}
Michael Fu et~al.
\newblock {\em Handbook of simulation optimization}, volume 216.
\newblock Springer, 2015.

\bibitem{GoluvanL89}
G.~H. Golub and C.~F. Van~Loan.
\newblock {\em Matrix Computations}.
\newblock Johns Hopkins University Press, Baltimore, second edition, 1989.

\bibitem{guyon2004result}
Isabelle Guyon, Steve Gunn, Asa Ben-Hur, and Gideon Dror.
\newblock Result analysis of the {NIPS} 2003 feature selection challenge.
\newblock In {\em Advances in neural information processing systems}, pages
  545--552, 2004.

\bibitem{lecun2010mnist}
Yann LeCun, Corinna Cortes, and Christopher~JC Burges.
\newblock {MNIST} handwritten digit database.
\newblock {\em AT\&T Labs [Online]. Available: http://yann. lecun.
  com/exdb/mnist}, 2010.

\bibitem{Lichman2013}
M.~Lichman.
\newblock {UCI} machine learning repository.
\newblock \url{http://archive.ics.uci.edu/ml}, 2013.

\bibitem{Martens10}
J.~Martens.
\newblock Deep learning via {H}essian-free optimization.
\newblock In {\em Proceedings of the 27th International Conference on Machine
  Learning (ICML)}, 2010.

\bibitem{mukherjee2013parallel}
Indraneel Mukherjee, Kevin Canini, Rafael Frongillo, and Yoram Singer.
\newblock Parallel boosting with momentum.
\newblock In {\em Joint European Conference on Machine Learning and Knowledge
  Discovery in Databases}, pages 17--32. Springer Berlin Heidelberg, 2013.

\bibitem{mybook}
Jorge Nocedal and Stephen Wright.
\newblock {\em Numerical {O}ptimization}.
\newblock Springer New York, 2 edition, 1999.

\bibitem{2014pasglyetal}
R.~Pasupathy, P.~W. Glynn, S.~Ghosh, and F.~Hashemi.
\newblock On sampling rates in stochastic recursions.
\newblock 2015.
\newblock Technical Report.

\bibitem{pilanci2015newton}
Mert Pilanci and Martin~J Wainwright.
\newblock Newton sketch: A linear-time optimization algorithm with
  linear-quadratic convergence.
\newblock {\em arXiv preprint arXiv:1505.02250}, 2015.

\bibitem{roosta2016sub1}
Roosta-Khorasani and Michael~W Mahoney.
\newblock Sub-sampled {Newton} methods {I}: Globally convergent algorithms.
\newblock {\em arXiv preprint arXiv:1601.04737}, 2016.

\bibitem{roosta2016sub}
Farbod Roosta-Khorasani and Michael~W Mahoney.
\newblock Sub-sampled {Newton} methods {II}: Local convergence rates.
\newblock {\em arXiv preprint arXiv:1601.04738}, 2016.

\bibitem{tropp_computational_2010}
Joel~A Tropp and Stephen~J Wright.
\newblock Computational methods for sparse solution of linear inverse problems.
\newblock {\em Proceedings of the {IEEE}}, 98(6):948--958, June 2010.

\bibitem{CINA}
Causality workbench team.
\newblock A marketing dataset.
\newblock \url{http://www.causality.inf.ethz.ch/data/CINA.html}, 09 2008.

\bibitem{xu2016sub}
Peng Xu, Jiyan Yang, Farbod Roosta-Khorasani, Christopher R{\'e}, and Michael~W
  Mahoney.
\newblock Sub-sampled {N}ewton methods with non-uniform sampling.
\newblock {\em arXiv preprint arXiv:1607.00559}, 2016.

\end{thebibliography}

\newpage
   \appendix
   \section{Additional Numerical Results} 
    Numerical results on the rest of the datasets listed in Table~5.1 are presented here.
      \begin{figure}[!htp]
       	\label{cina-ex1} 
       	\includegraphics[width=1.0\linewidth]{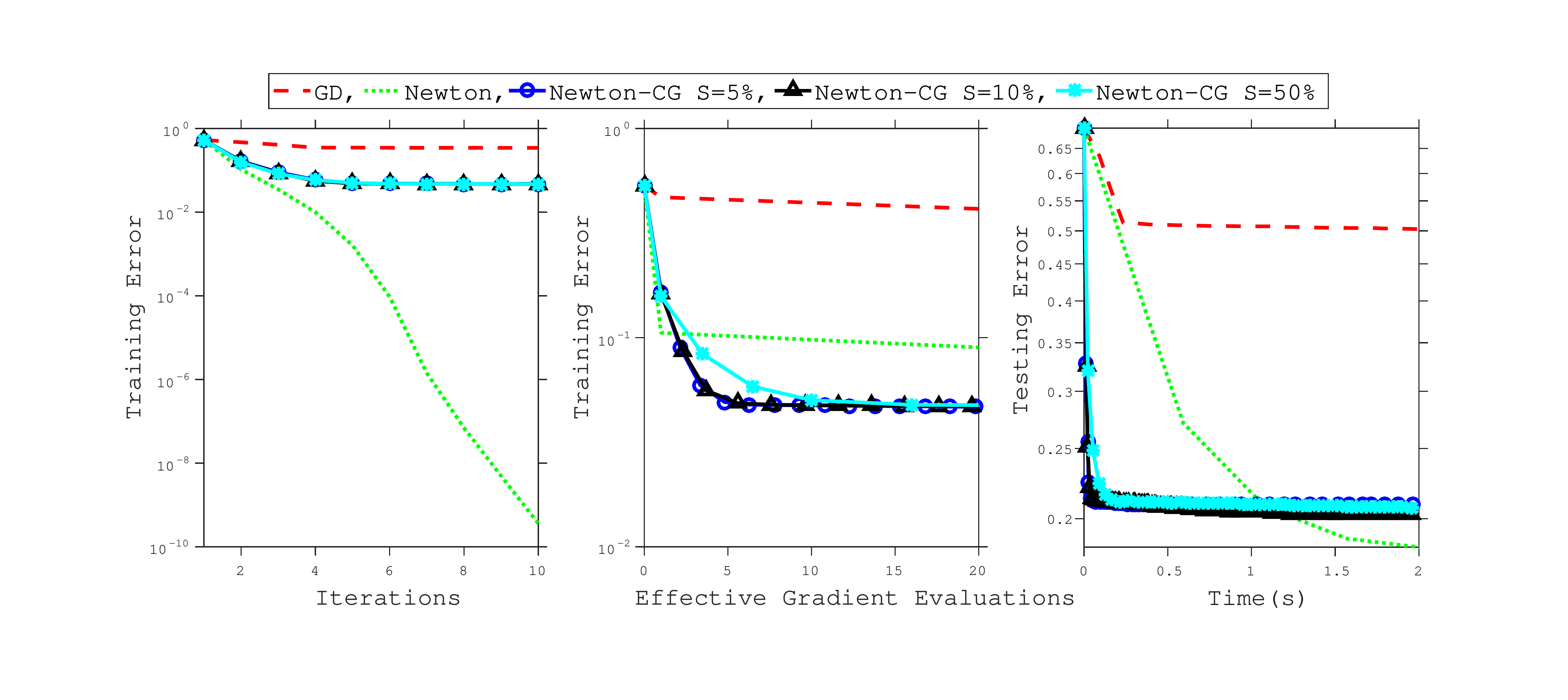}
       	\caption{ {\bf Cina Dataset}: Performance of the inexact subsampled Newton method (Newton-CG), using three values of the sample size, and of the GD and Newton methods. Left: Training Error vs. Iterations; Middle: Training Error vs. Effective Gradient Evaluations; Right: Testing Objective vs Time.}
      \end{figure}

       \begin{figure}[!htp]
       	\label{Cina-ex2} 
       	\includegraphics[width=1\linewidth]{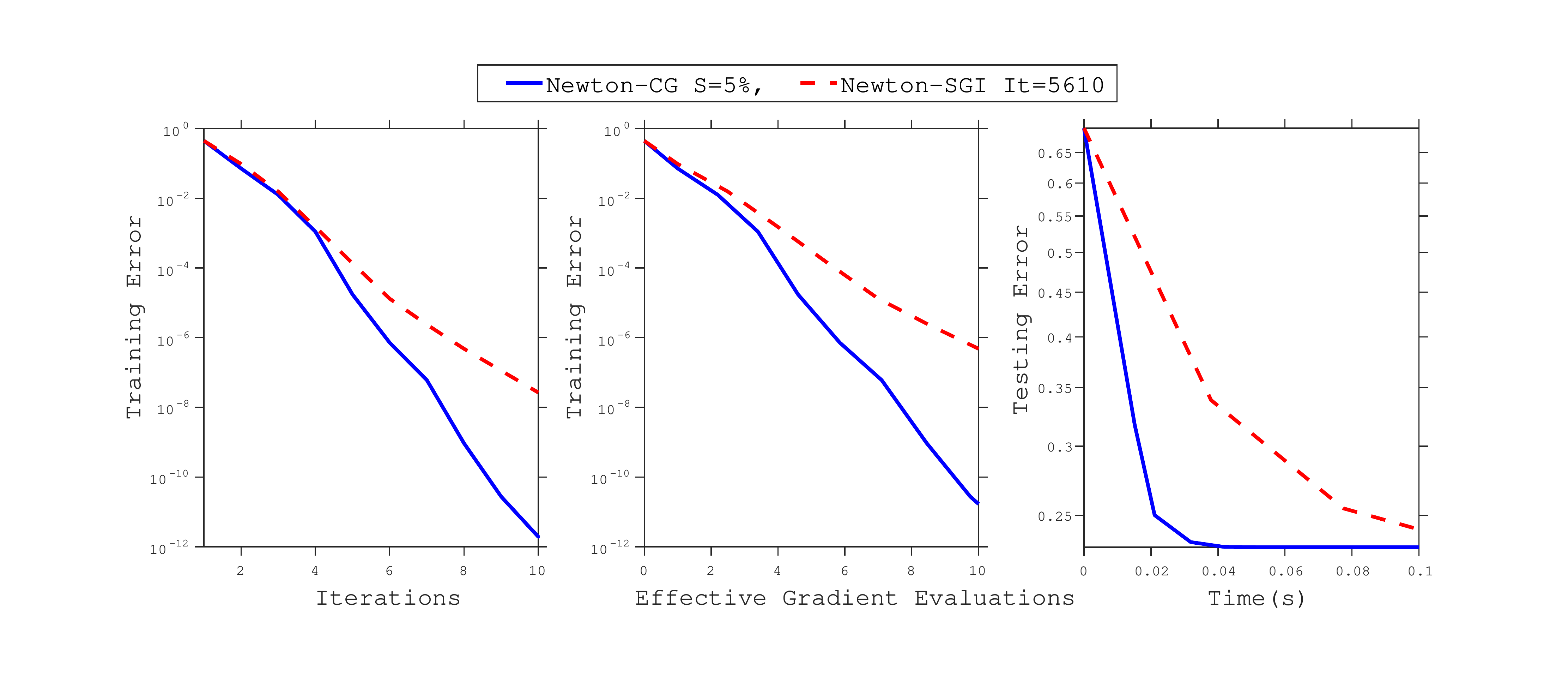}
       	\caption{ {\bf Cina Dataset (scaled)}: Comparison of Newton-CG with Newton-SGI. Left: Training Error vs Iterations; Middle: Training Error vs. Effective Gradient Evaluations; Right: Testing Error vs. Time.  The number of SGI iterations is determined through $It=|S|r$.}
       \end{figure}
    
           \begin{figure}[!htp]
              	\label{cina-ex3} 
              	\includegraphics[width=1\linewidth]{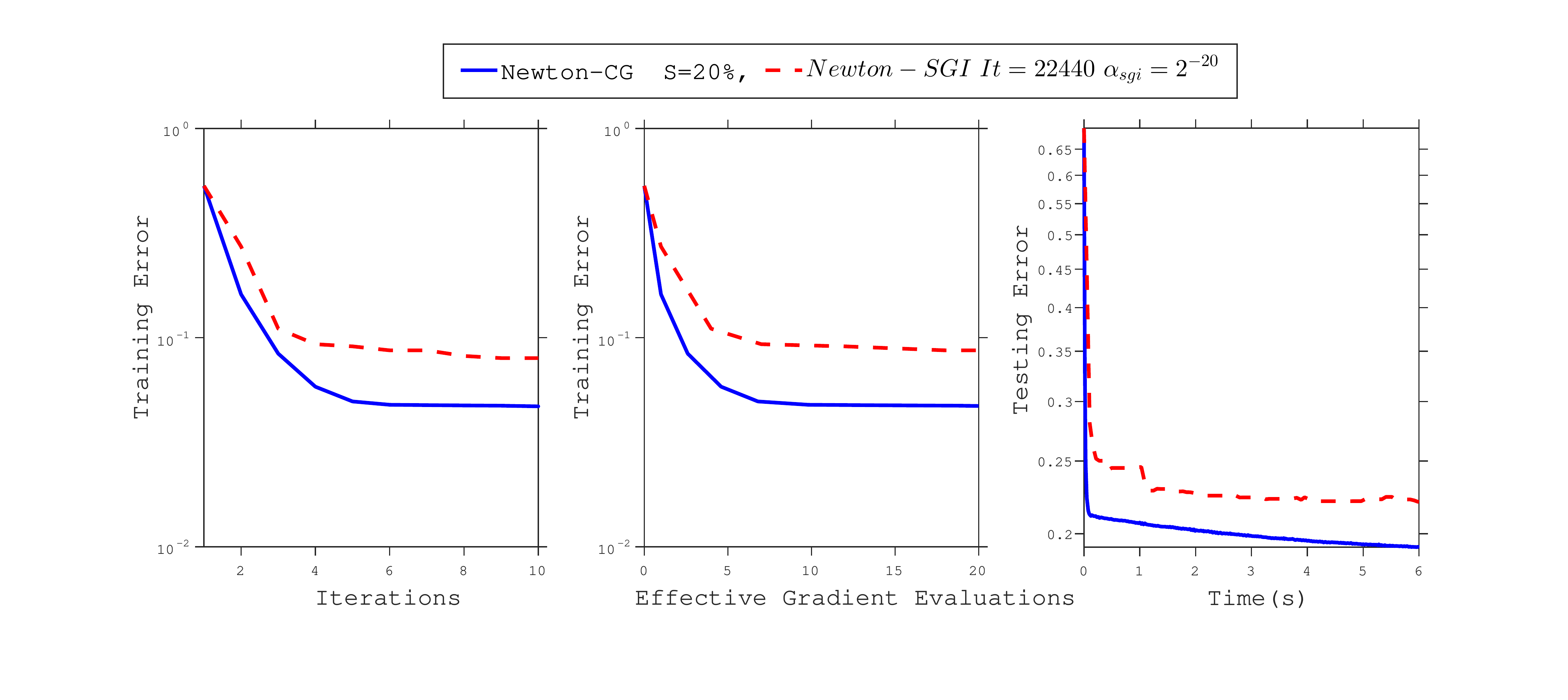}
              	\caption{ {\bf Cina Dataset (unscaled)}: Comparison of Newton-CG with Newton-SGI. Left: Training Error vs. Iterations; Middle: Training Error vs. Effective Gradient Evaluations; Right: Testing Objective vs. Time.}
           \end{figure}
           \newpage
           
     \begin{figure}[!htp]
     	\label{mushrooms-ex1} 
     	\includegraphics[width=1.0\linewidth]{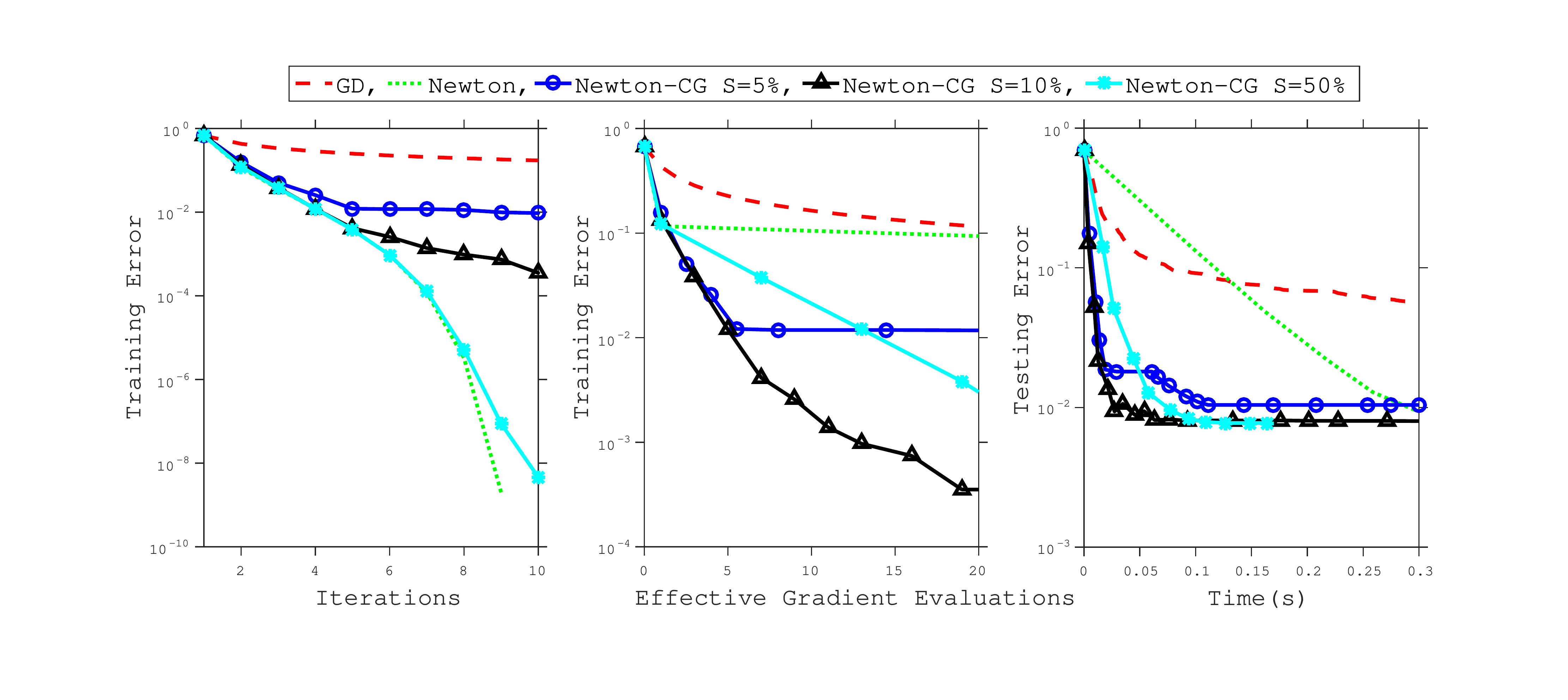}
     	\caption{ {\bf Mushrooms Dataset}: Performance of the inexact subsampled Newton method (Newton-CG), using three values of the sample size, against two other methods. Left: Training Error vs. Iterations; Middle: Training Error vs. Effective Gradient Evaluations; Right: Testing Objective vs. Time.}
     \end{figure}
     
     \begin{figure}[!htp]
     	\label{mushrooms-ex2} 
     	\includegraphics[width=1\linewidth]{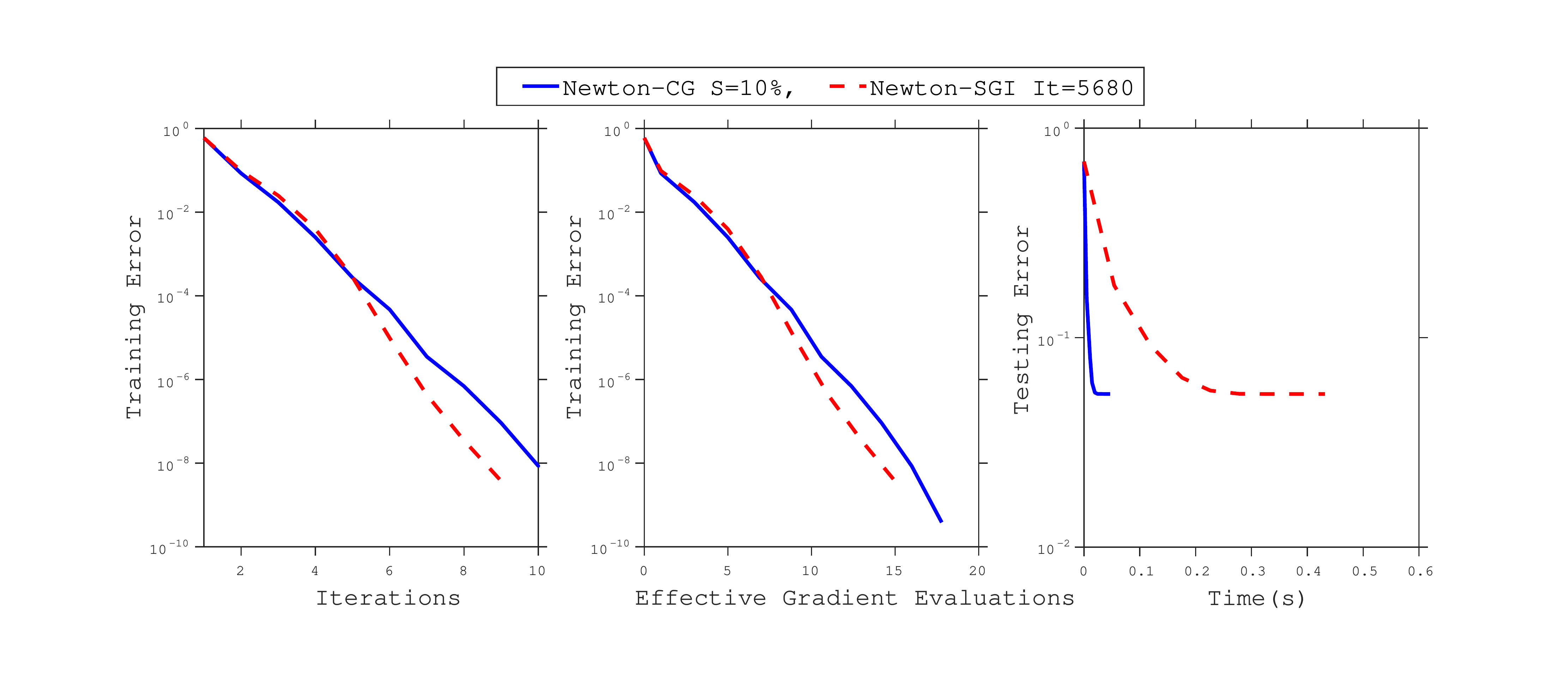}
     	\caption{ {\bf Mushrooms Dataset (scaled)}: Comparison of Newton-CG with Newton-SGI. Left: Training Error vs. Iterations; Middle: Training Error vs. Effective Gradient Evaluations; Right: Testing Error vs. Time.}
     \end{figure}
     
     \begin{figure}[!htp]
     	\label{mushrooms-ex3} 
     	\includegraphics[width=1\linewidth]{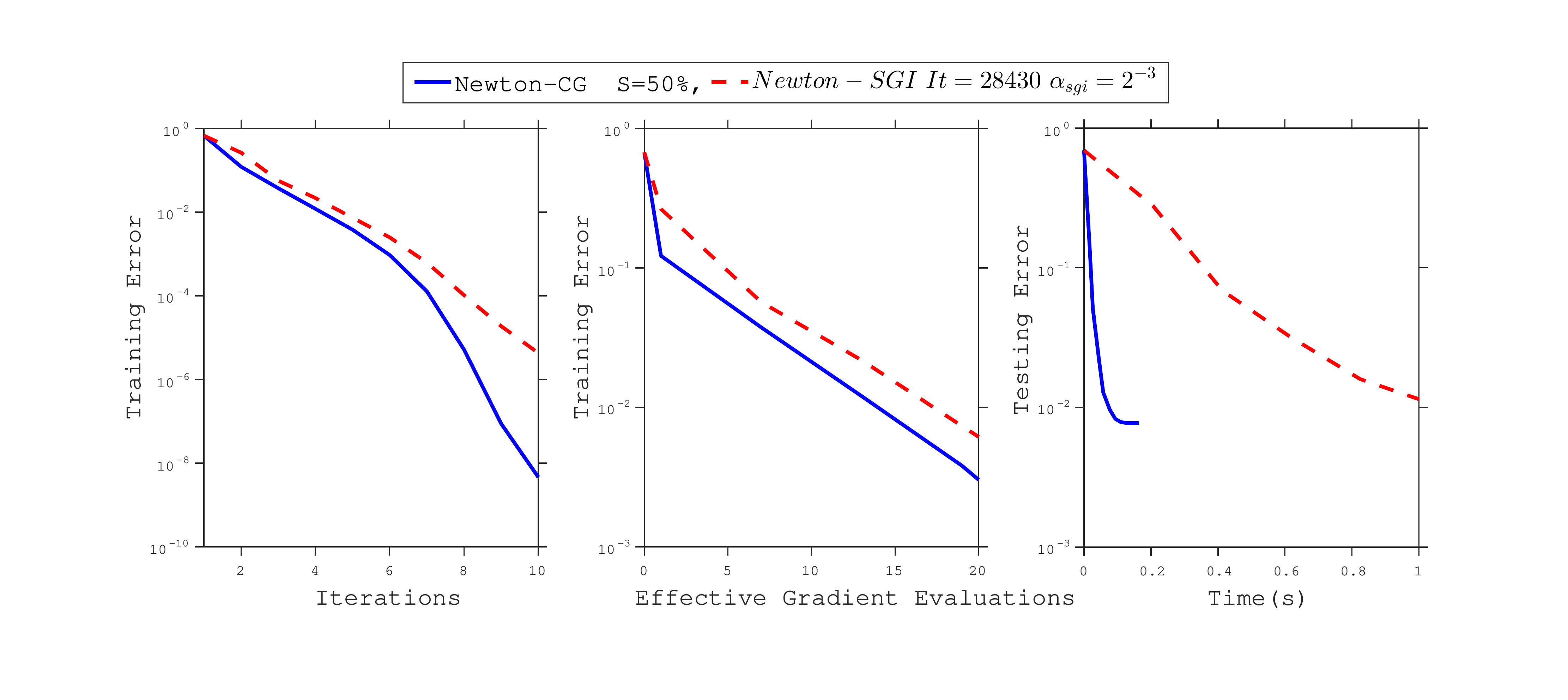}
     	\caption{ {\bf Mushrooms Dataset (unscaled)}: Comparison of Newton-CG with Newton-SGI. Left: Training Error vs. Iterations; Middle: Training Error vs. Effective Gradient Evaluations; Right: Testing Objective vs. Time.}
     \end{figure}
     \newpage
     \begin{figure}[!htp]
     	\label{mnist-ex1} 
     	\includegraphics[width=1.0\linewidth]{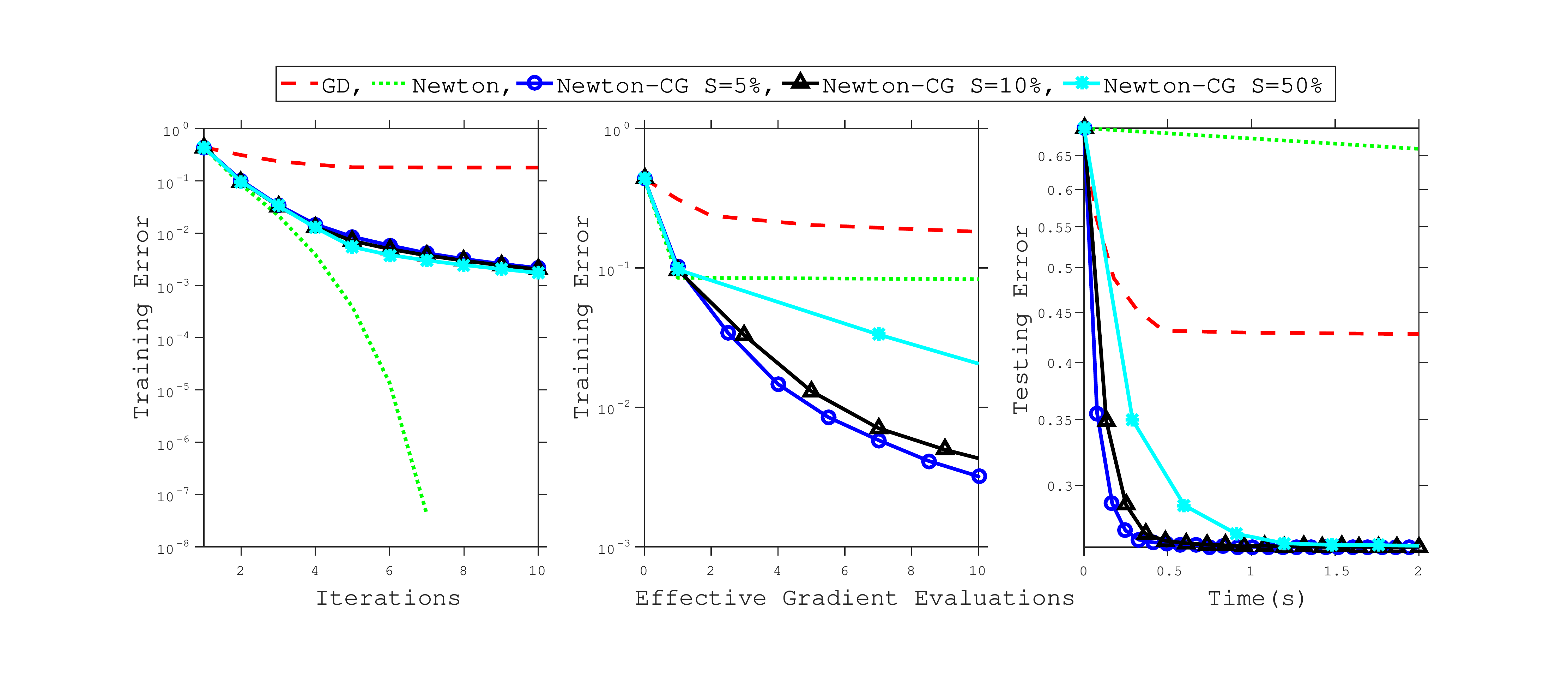}
     	\caption{ {\bf MNIST Dataset}: Performance of the inexact subsampled Newton method (Newton-CG), using three values of the sample size, and of the GD and Newton methods. Left: Training Error vs. Iterations; Middle: Training Error vs. Effective Gradient Evaluations; Right: Testing Objective vs. Time.}
     \end{figure}
     
     \begin{figure}[!htp]
     	\label{mnist-ex2} 
     	\includegraphics[width=1\linewidth]{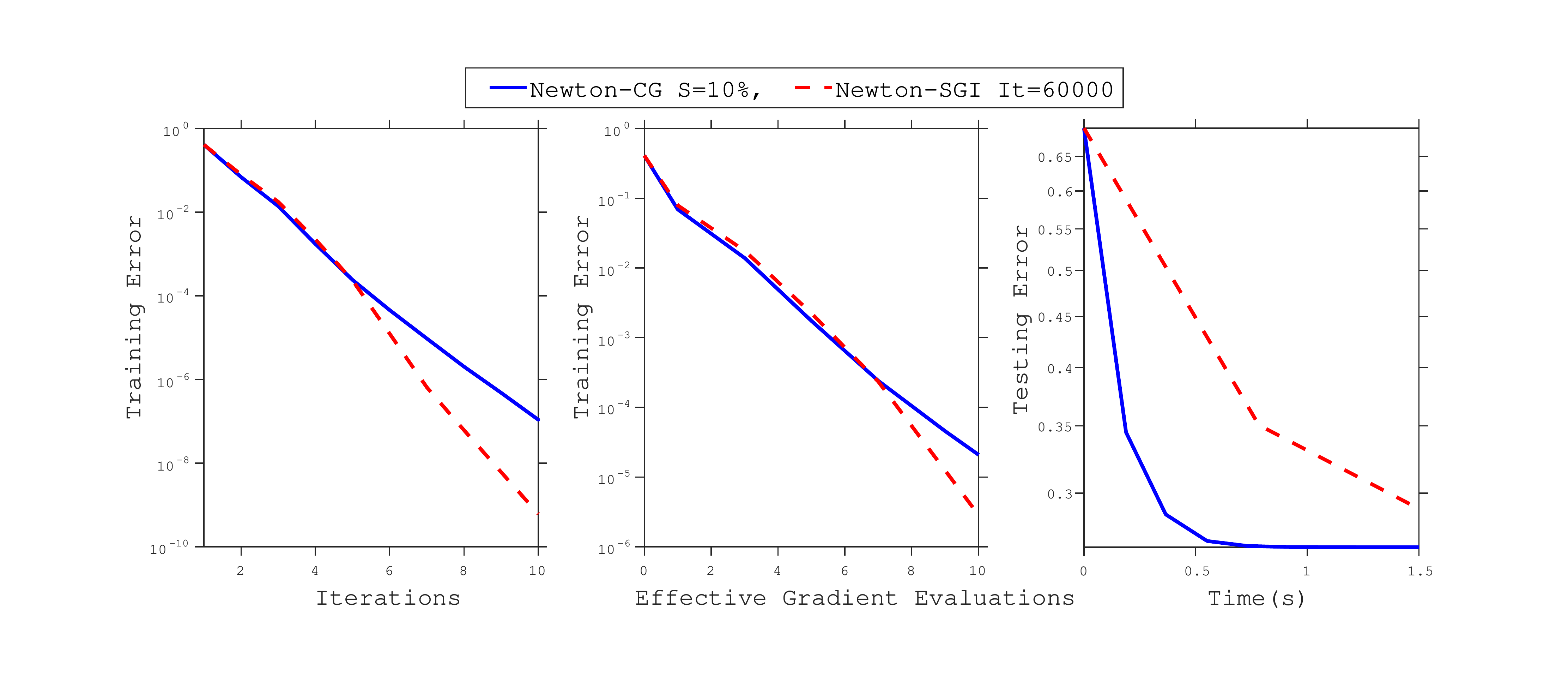}
     	\caption{ {\bf MNIST Dataset (scaled)}: Comparison of Newton-CG with Newton-SGI. Left: Training Error vs. Iterations; Middle: Training Error vs. Effective Gradient Evaluations; Right: Testing Error vs. Time.}
     \end{figure}
     
     \begin{figure}[!htp]
     	\label{mnist-ex3} 
     	\includegraphics[width=1\linewidth]{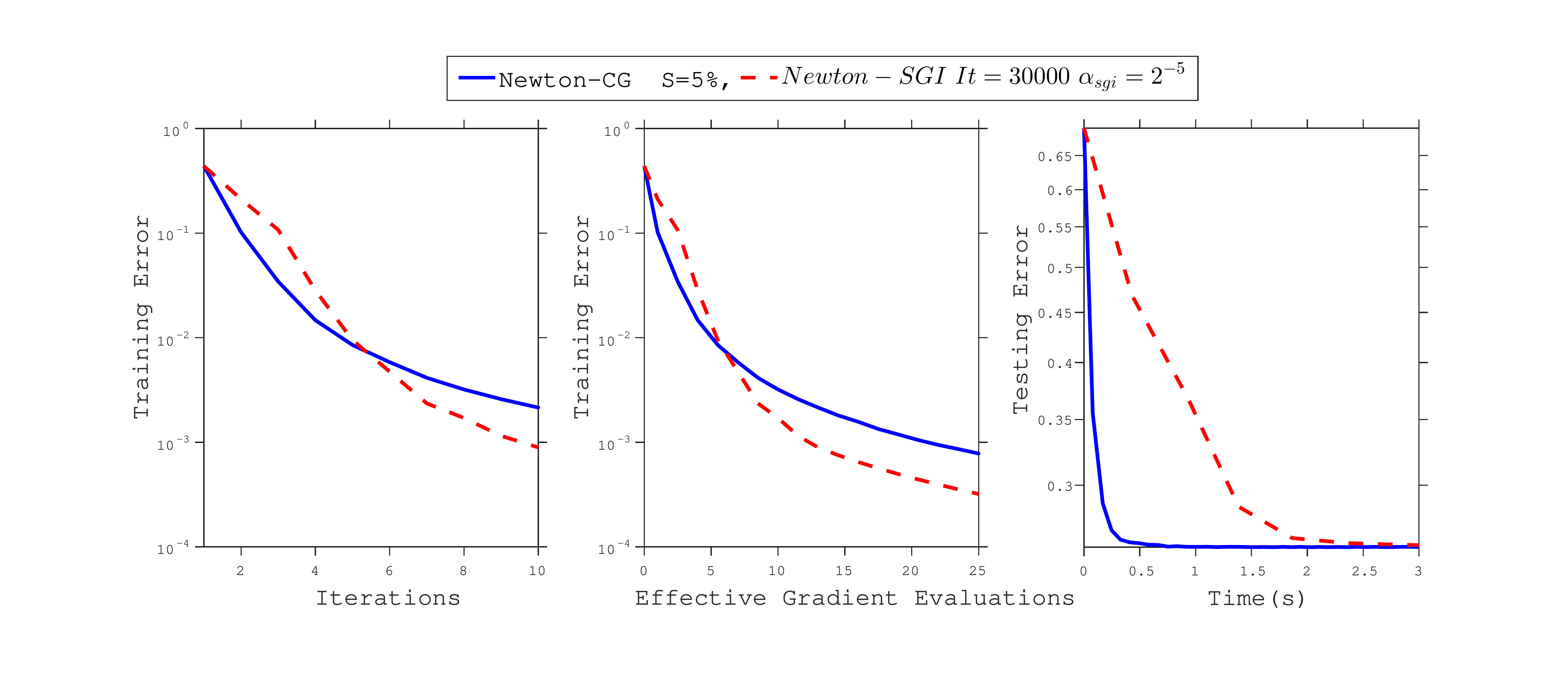}
     	\caption{ {\bf MNIST Dataset (unscaled)}: Comparison of Newton-CG with Newton-SGI. Left: Training Error vs. Iterations; Middle: Training Error vs. Effective Gradient Evaluations; Right: Testing Objective vs. Time.}
     \end{figure}
     
     {\bf Note:} The MNIST Datset has been used for binary classification of digits into even and odd.
     \newpage
     \begin{figure}[!htp]
     	\label{gisette-ex1} 
     	\includegraphics[width=1.0\linewidth]{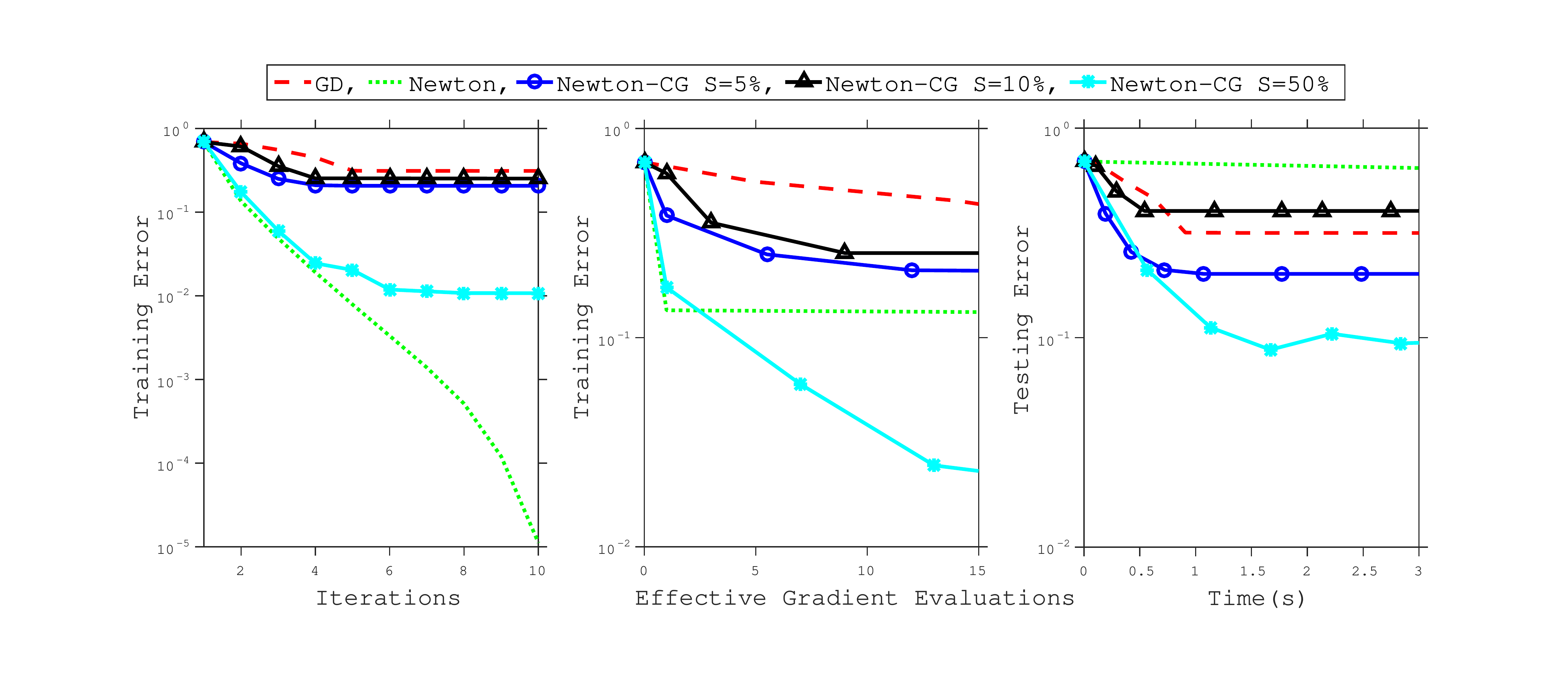}
     	\caption{ {\bf Gisette Dataset}: Performance of the inexact subsampled Newton method (Newton-CG), using three values of the sample size, and of the GD and Newton methods. Left: Training Error vs. Iterations; Middle: Training Error vs. Effective Gradient Evaluations; Right: Testing Objective vs. Time.}
     \end{figure}
     
     \begin{figure}[!htp]
     	\label{gisette-ex2} 
     	\includegraphics[width=1\linewidth]{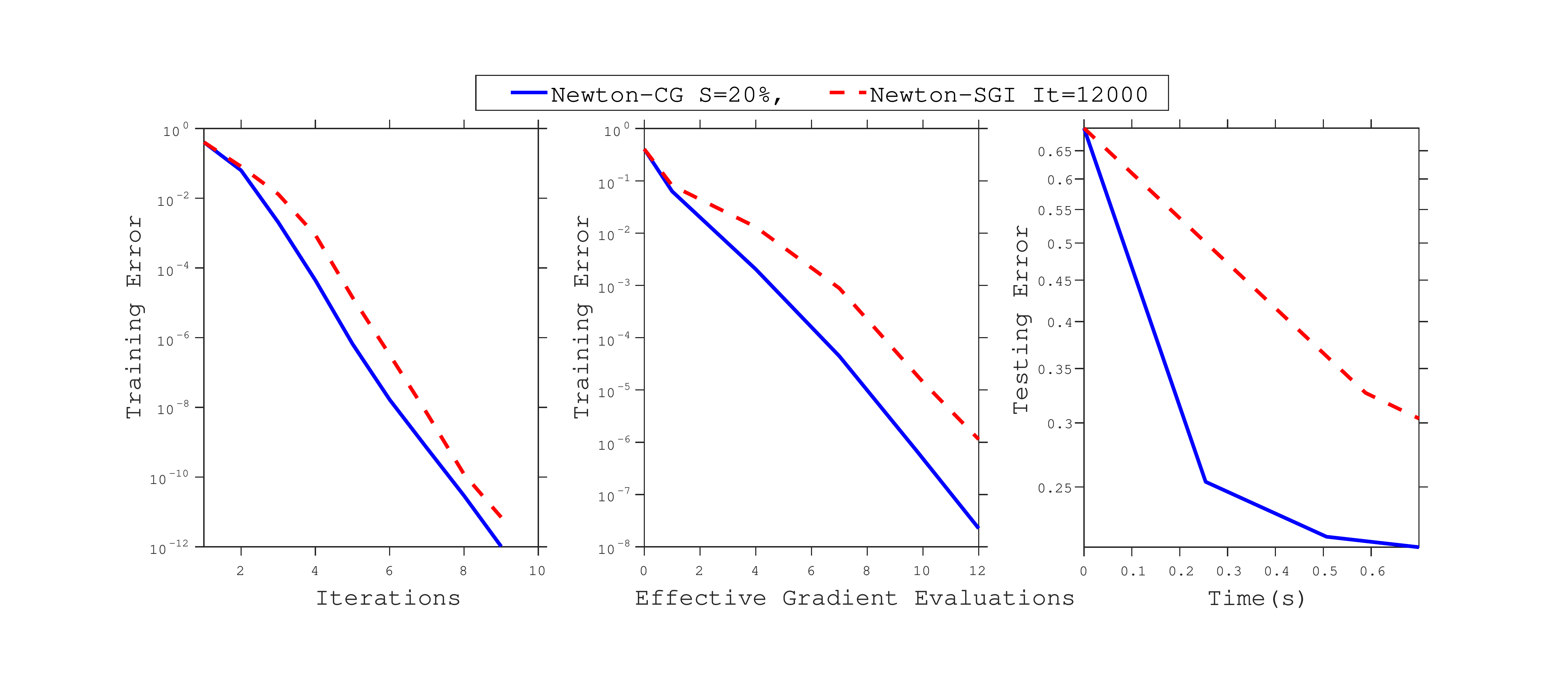}
     	\caption{ {\bf Gisette Dataset (scaled)}: Comparison of Newton-CG with Newton-SGI. Left: Training Error vs. Iterations; Middle: Training Error vs. Effective Gradient Evaluations; Right: Testing Error vs. Time.}
     \end{figure}
     
     \begin{figure}[!htp]
     	\label{gisette-ex3} 
     	\includegraphics[width=1\linewidth]{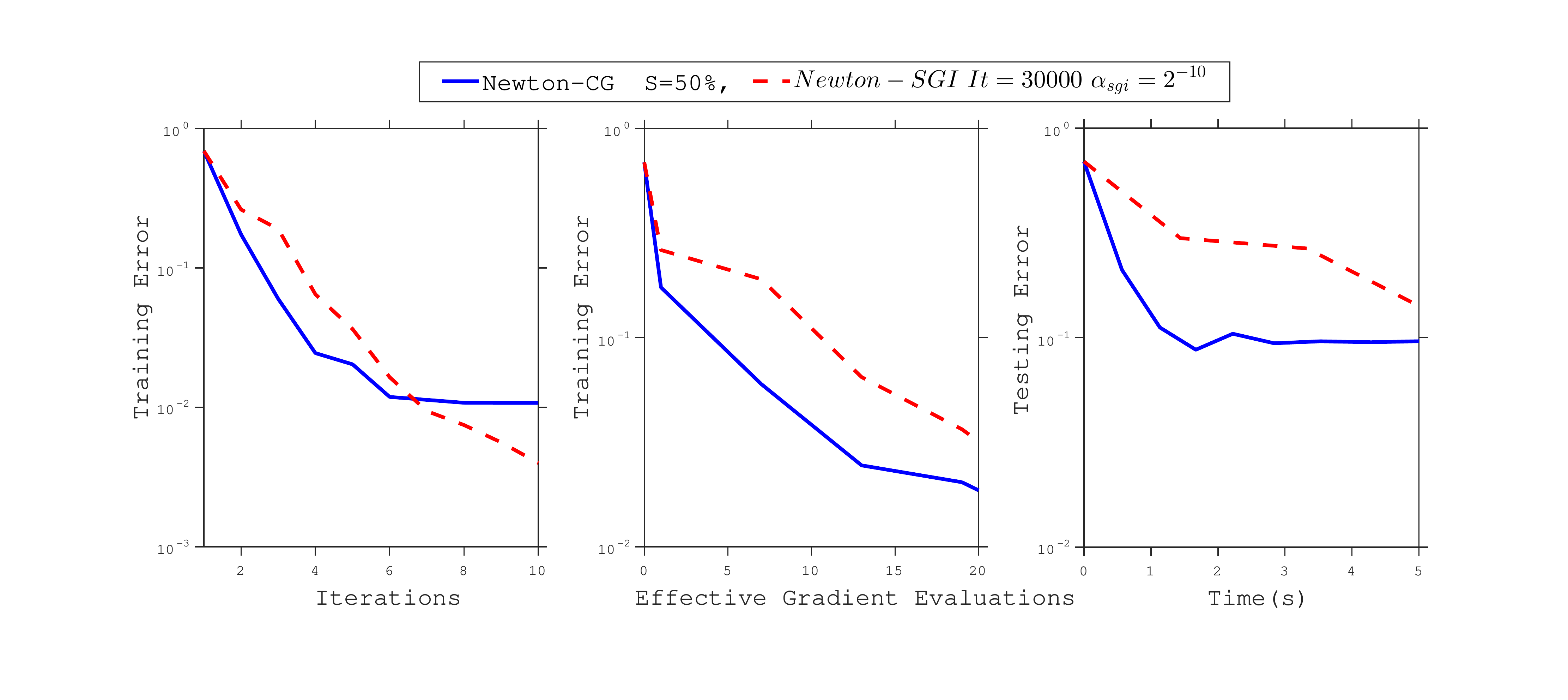}
     	\caption{ {\bf Gisette Dataset (unscaled)}: Comparison of Newton-CG with Newton-SGI. Left: Training Error vs. Iterations; Middle: Training Error vs. Effective Gradient Evaluations; Right: Testing Objective vs. Time.}
     \end{figure}
     \newpage
     
     \begin{figure}[!htp]
     	\label{covtype-ex1} 
     	\includegraphics[width=1.0\linewidth]{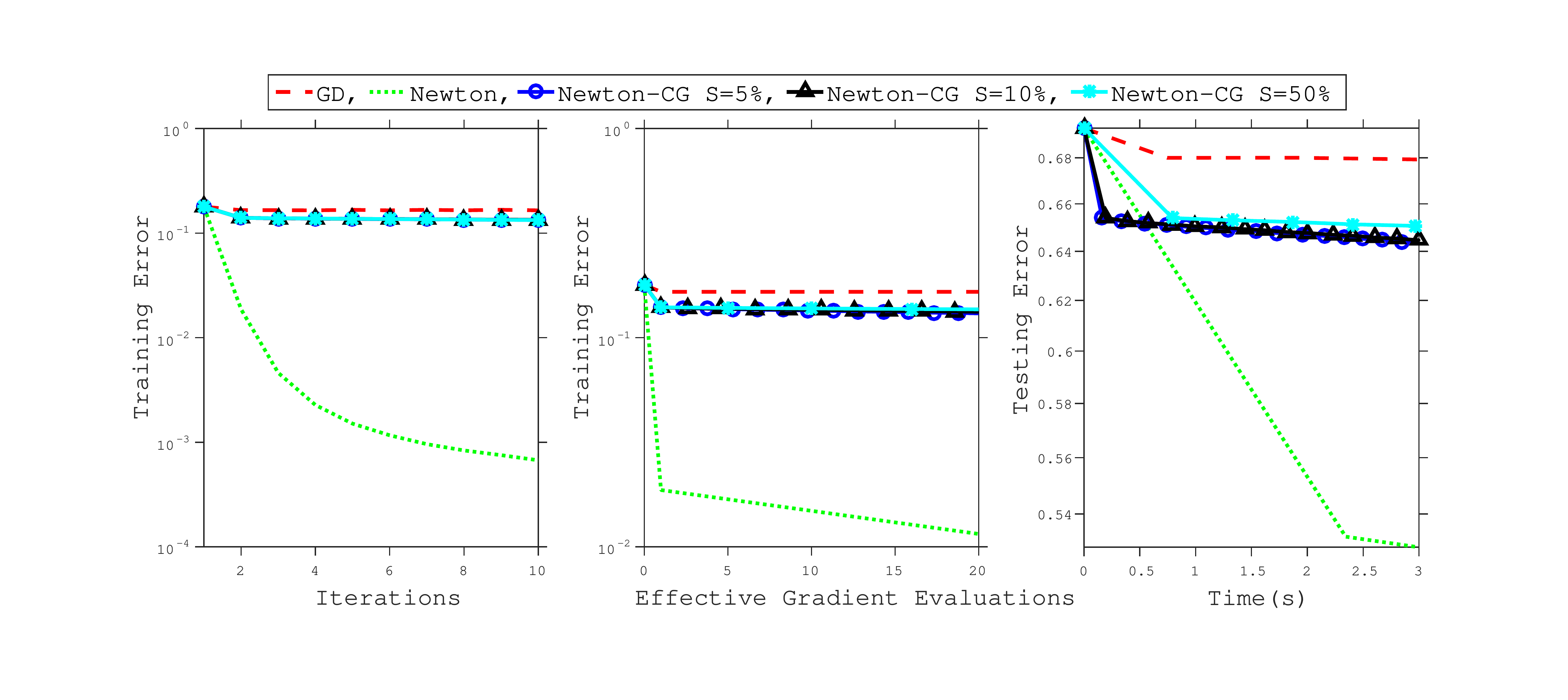}
     	\caption{ {\bf Covertype Dataset}: Performance of the inexact subsampled Newton method (Newton-CG), using three values of the sample size, and of the GD and Newton methods. Left: Training Error vs. Iterations; Middle: Training Error vs. Effective Gradient Evaluations; Right: Testing Objective vs. Time.}
     \end{figure}
     
     \begin{figure}[!htp]
     	\label{covtype-ex2} 
     	\includegraphics[width=1\linewidth]{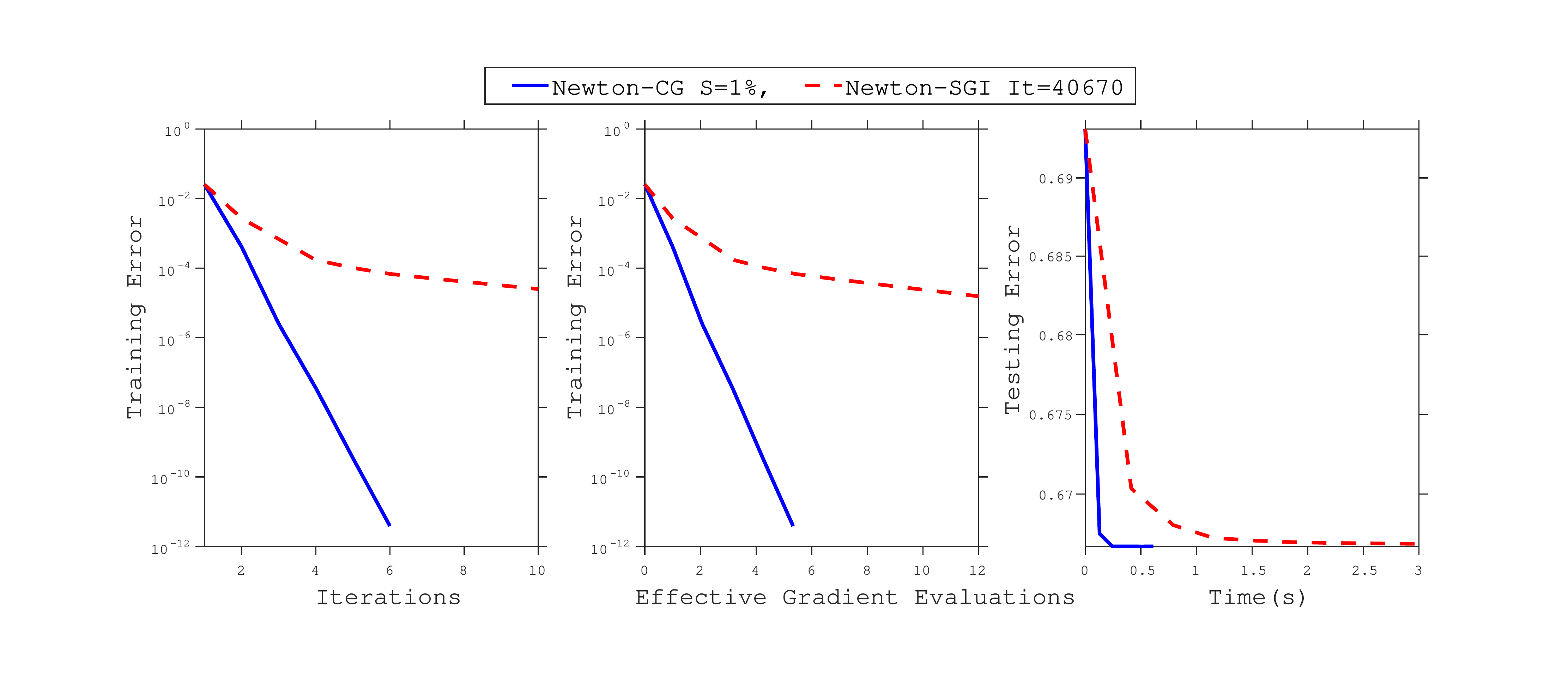}
     	\caption{ {\bf Covertype Dataset (scaled)}: Comparison of Newton-CG with Newton-SGI. Left: Training Error vs. Iterations; Middle: Training Error vs. Effective Gradient Evaluations; Right: Testing Error vs. Time.}
     \end{figure}  
     
     Note: Newton-SGI method for covertype dataset (unscaled) didn't converge for any of the steplength $\alpha_{sgi} \in \{2^{-20},...,2^3\}$

    \end{document}